\documentclass[10pt]{amsart} 

\usepackage[utf8]{inputenc} 

\usepackage{lipsum}

\usepackage{geometry} 
\usepackage{graphicx} 

\usepackage[english]{babel}

\usepackage{booktabs} 
\usepackage{array} 
\usepackage{paralist} 
\usepackage{verbatim} 
\usepackage{subfig} 

\usepackage{caption}

\usepackage{amsfonts,amsmath,amssymb,bbm,amsthm,amsbsy}

\usepackage{mathtools}

\newenvironment{frcseries}{\fontfamily{frc}\selectfont}{}

\makeatletter							
\newtheorem*{rep@theorem}{\rep@title}
\newcommand{\newreptheorem}[2]{%
\newenvironment{rep#1}[1]{%
 \def\rep@title{#2 \ref{##1}}%
 \begin{rep@theorem}}%
 {\end{rep@theorem}}}
\makeatother

\newtheorem{thm}{Theorem}[section]
\newtheorem{lemma}[thm]{Lemma}
\newtheorem{prop}[thm]{Proposition}
\newtheorem{corr}[thm]{Corollary}
\newtheorem{claim}[thm]{Claim}

\newtheorem*{thm*}{Theorem}
\newtheorem*{lemma*}{Lemma}
\newtheorem*{prop*}{Proposition}
\newtheorem*{corr*}{Corrolary}
\newtheorem*{claim*}{Claim}

\theoremstyle{remark}
\newtheorem{rmk}[thm]{Remark}

\newtheorem*{rmk*}{Remark}
\newtheorem*{conj*}{Conjecture}
\newtheorem*{quest*}{Question}

\theoremstyle{definition}
\newtheorem{defn}[thm]{Definition}
\newtheorem{exmp}[thm]{Example}

\newtheorem*{defn*}{Definition}
\newtheorem*{exmp*}{Example}

\newreptheorem{theorem}{Theorem}
\newreptheorem{corollary}{Corollary}
\newreptheorem{proposition}{Proposition}

\usepackage{fancyhdr} 
\usepackage{tikz}
\usepackage{tikz-cd}

\usepackage{IEEEtrantools}

\newenvironment{equ*}[1]{\begin{IEEEeqnarray*}{#1}}{\end{IEEEeqnarray*}}

\newcommand{\R}{\mathbb{R}}
\newcommand{\F}{\mathbb{F}}

\newcommand{\Z}{\mathbb{Z}}
\newcommand{\N}{\mathbb{N}}

\newcommand{\Scl}{\mathcal{S}}

\newcommand{\att}{\texttt{a}}

\newcommand{\at}{\texttt{a}}
\newcommand{\bt}{\texttt{b}}
\newcommand{\xt}{\texttt{x}}
\newcommand{\yt}{\texttt{y}}
\newcommand{\zt}{\texttt{z}}

\newcommand{\sign}{\textrm{sign}}

\newcommand{\Arm}{\mathrm{A}}
\newcommand{\Vrm}{\mathrm{V}}
\newcommand{\Erm}{\mathrm{E}}
\newcommand{\Trm}{\mathrm{T}}
\newcommand{\Homeorm}{\mathrm{Homeo}}
\newcommand{\eurm}{\mathrm{eu}}
\newcommand{\Hrm}{\mathrm{H}}

\newcommand{\Acl}{\mathcal{A}}

\newcommand{\To}{\text{[T1]}}
\newcommand{\Toa}{\text{[T1$\at$]}}
\newcommand{\Tob}{\text{[T1$\bt$]}}
\newcommand{\Tt}{\text{[T2]}}
\newcommand{\Tta}{\text{[T2$\at$]}}
\newcommand{\Ttb}{\text{[T2$\bt$]}}

\newcommand{\cl}{\mathrm{cl}}
\newcommand{\scl}{\mathrm{scl}}

\newcommand{\inj}{\hookrightarrow}

\newcommand{\col}{\colon}

\author{Nicolaus Heuer}
\date{\today}
\title{Gaps in scl for Amalgamated Free Products and RAAGs}
\address{Department of Mathematics\\
  University of Oxford}
\email[N.~Heuer]{heuer@maths.ox.ac.uk}

\begin{document}

\begin{abstract}
We develop a new criterion to tell if a group $G$ has the maximal gap of $1/2$ in stable commutator length (scl).
For amalgamated free products $G = A \star_C B$ we show that every element $g$ in the commutator subgroup of $G$ which does not conjugate into $A$ or $B$ satisfies $\scl(g) \geq 1/2$, provided that $C$ embeds as a \emph{left relatively convex} subgroup in both $A$ and $B$. 
We deduce from this that every non-trivial element $g$ in the commutator subgroup of a right-angled Artin group $G$ satisfies $\scl(g) \geq 1/2$. This bound is sharp and is inherited by all fundamental groups of special cube complexes.

We prove these statements by constructing explicit extremal homogeneous quasimorphisms $\bar{\phi} \col G \to \R$ satisfying $\bar{\phi}(g) \geq 1$ and $D(\bar{\phi})\leq 1$. Such maps were previously unknown, even for non-abelian free groups. 
For these quasimorphisms $\bar{\phi}$ there is an action $\rho \col G \to \mathrm{Homeo}^+(S^1)$ on the circle such that $[\delta^1 \bar{\phi}]=\rho^*\eurm^\R_b \in \Hrm^2_b(G,\R)$, for $\eurm^\R_b$ the real bounded Euler class.
\end{abstract}

\maketitle

\section{Introduction}
For a group $G$ let $G'$ be the commutator subgroup. For an element $g \in G'$ the \emph{commutator length} ($\cl(g)$) denotes the minimal number of commutators needed to express $g$ as their product.
We define the \emph{stable commutator length (scl)} via $\scl(g) = \lim_{n \to \infty} \cl(g^n)/n$.

Stable commutator length is well studied and has geometric meaning: 
Let $X$ be a topological space, let $\gamma$ be a loop in $X$ and let $[\gamma]$ be the conjugacy class in $\pi_1(X)$ corresponding to $\gamma$. 
Then both $\cl([\gamma])$ and $\scl([\gamma])$ measure the minimal complexity of an orientable surface needed 
to bound $\gamma$. The theory of these invariants is developed by Calegari in \cite{calegari:scl}.
A group $G$ is said to have a \emph{gap in stable commutator length} if there is a constant $C>0$ such that either $\scl(g) = 0$ or $\scl(g) \geq C$ for every non-trivial $g \in G'$. If $G$ is non-abelian, such a constant necessarily satisfies $C \leq 1/2$.
Similarly we may define gaps in scl for classes of groups. Many classes of ``negatively curved'' groups have a gap in scl; see Subsection \ref{subsec:spectral gaps of scl}. 

A common way of establishing gaps in $\scl$ is by constructing \emph{quasimorphisms} and using \emph{Bavard's Duality Theorem} (see \cite{bavard}): For an element $g \in G'$,
\[
\scl(g)  = \sup_{\bar{\phi} \in \mathcal{Q}(G)} \frac{\bar{\phi}(g)}{2 D(\bar{\phi}) }
\] 
where $\mathcal{Q}(G)$ is the space of \emph{homogeneous quasimorphisms} and $D(\bar{\phi})$ is the \emph{defect of $\bar{\phi}$}; see Subsection \ref{subsec:quasimorphisms and Bavard's Duality} for the definitions and the precise statement.
Though it is known that for every element $g \in G'$ the supremum in Bavard's Duality Theorem is obtained by so-called \emph{extremal quasimorphism} these maps are only known explicitly in special cases and are hard to construct;
see \cite{calegari:extremal} and \cite{calegari:isometric}. 

In the first part of this paper, we will construct a family of extremal quasimorphisms on non-abelian free groups. Let $\F_2 = \langle \at, \bt \rangle$ be the free group on generators $\at$ and $\bt$ and let $w \in \F_2$ be such that it does not conjugate into $\langle \at \rangle$ or $\langle \bt \rangle$. Then we will construct a homogeneous quasimorphism $\bar{\phi}$ such that 
$\bar{\phi}(w) \geq 1$ and $D(\bar{\phi})\leq 1$. This realises the well-known gap of $1/2$ in the case of non-abelian free groups.
Our approach is as follows: instead of constructing more complicated quasimorphisms $\bar{\phi}$ we first ``simplify'' the element $w$.

This simplification is formalised by functions $\Phi \col G \to \Acl \subset \F_2$, called \emph{letter-quasimorphisms}; see Definition \ref{defn:letter quasihomomorphism}. 
Here $\Acl$ denotes the set of \emph{alternating words} in $\F_2 = \langle \at, \bt \rangle$ with the generators $\at$ and $\bt$. These are words where each letter alternates between $\{ \at, \at^{-1} \}$ and $\{ \bt, \bt^{-1} \}$.
Letter-quasimorphisms are a special case of quasimorphisms between arbitrary groups defined by Hartnick--Schweitzer \cite{hartnick-schweitzer}. After this simplification, the extremal quasimorphisms on $G$ are obtained 
by pulling back most basic quasimorphisms $\F_2 \to \R$ via such letter-quasimorphisms $G \to \Acl \subset \F_2$. We further deduce that such quasimorphisms are induced by a circle action $\rho \col G \to \Homeorm^+(S^1)$ by examining the defect and using Theorem \ref{thm:ghys} due to Ghys; see also \cite{ghys}.
We show:
\begin{reptheorem}{thm:main}
Let $G$ be a group, $g \in G$ and suppose that there is a letter-quasimorphism $\Phi \col G \to \Acl$ such that $\Phi(g)$ is non-trivial and $\Phi(g^n) = \Phi(g)^n$ for all $n \in \N$.
Then there is an explicit homogeneous quasimorphism $\bar{\phi} \col G \to \R$ such that $\bar{\phi}(g) \geq 1$ and $D(\bar{\phi}) \leq 1$.

If $G$ is countable then there is an action $\rho \col G \to \mathrm{Homeo}^+(S^1)$ such that $[\delta^1 \bar{\phi}]=\rho^*\eurm^\R_b \in \Hrm^2_b(G,\R)$, for $\eurm^\R_b$ the real bounded Euler class.
\end{reptheorem}
By Bavard's Duality Theorem it is immediate that if such an element $g$ additionally lies in $G'$, then $\scl(g) \geq 1/2$.
We state Theorem \ref{thm:main} separately as it may also be applied in other cases than the ones presented in this paper; see Remark \ref{rmk:chen-heuer}.
Many groups $G$ have the property that for any element $g \in G'$ there is a letter-quasimorphism $\Phi_g \col G \to \Acl$ such that $\Phi_g(g^n) = \Phi_g(g)^n$ where $\Phi_g(g) \in \Acl$ is non-trivial. We will see that residually free groups and right-angled Artin groups have this property. Note the similarities of this property with being \emph{residually free}; see Remark \ref{rmk:criterion for gaps}.  

In the second part of this paper we apply Theorem \ref{thm:main} to amalgamated free products using left-orders. A subgroup $H < G$ is called \emph{left-relatively convex} if there is an order on the left cosets $G/H$ which is invariant under left multiplication by $G$. We will construct letter-quasimorphisms $G \to \Acl \subset \F_2$ using the sign of these orders. We deduce:
\begin{reptheorem}{thm:amalgamation}
Let $A, B, C$ be groups, $\kappa_A \col C \inj A$ and $\kappa_B \col C \inj B$ injections and suppose both $\kappa_A(C) < A$ and $\kappa_B(C) < B$ are left-relatively convex. If $g \in A \star_C B$ does not conjugate into one of the factors then there is a homogeneous quasimorphism $\bar{\phi} \col A \star_C B \to \R$ such that $\bar{\phi}(g) \geq 1$ and
$D(\bar{\phi}) \leq 1$.
If $G$ is countable then there is an action $\rho \col G \to \mathrm{Homeo}^+(S^1)$ such that $[\delta^1 \bar{\phi}]=\rho^*\eurm^\R_b \in \Hrm^2_b(G,\R)$, for $\eurm^\R_b$ the real bounded Euler class.
\end{reptheorem}
It is possible to generalise Theorem \ref{thm:amalgamation} to graphs of groups; see Remark \ref{rmk:chen-heuer}.
Again by Bavard's Duality Theorem we infer that any such $g$ which also lies in the commutator subgroup satisfies $\scl(g) \geq 1/2$. We apply this to right-angled Artin groups using the work of \cite{convexsub}. This way we prove:
\begin{reptheorem}{thm:raags and scl}
Every non-trivial element $g \in G'$ in the commutator subgroup of a right-angled Artin group $G$ satisfies $\scl(g) \geq 1/2$. This bound is sharp.
\end{reptheorem}
This is an improvement of the bound previously found in \cite{raags1} and \cite{raags2} who deduced a general bound of $1/24$ and a bound of $1/20$ if the right-angled Artin group is two dimensional.
Every subgroup of a right-angled Artin group will inherit this bound. Such groups are now known to be an extremely rich class, following the theory of special cube complexes. See \cite{wise}, \cite{haglund-wise}, \cite{agol}, \cite{martin1} and \cite{martin2}.
Stable commutator length may serve as an invariant to distinguish virtually special from special cube complexes.

We collect some properties of the quasimorphisms constructed in this paper.
\begin{itemize}
\item 
The quasimorphisms are induced by circle actions $\rho \col G \to \Homeorm^+(S^1)$ even though we do not construct the explicit action $\rho$. In particular, for every $e \not = g \in F'$ where $F$ is a non-abelian free group and $\scl(g) = 1/2$ there is an \emph{extremal} quasimorphism $\bar{\phi} \col F \to \R$ induced by a circle action.
It is unknown if 
for an arbitrary element $g \in F'$ there is an action of $F$ on the circle such that the induced quasimorphism is extremal with respect to $g$. 
\item There are relatively few quasimorphisms needed to obtain the $1/2$ bound in Theorem \ref{thm:raags and scl}. Let $G$ be a right-angled Artin group.
Analysis of the constructions show that there is a sequence $\Scl_N \subset \mathcal{Q}(G)$ of nested sets  of homogeneous quasimorphisms such that for every non-trivial cyclically reduced element $g$ of length less than $N$ there is some $\bar{\phi} \in \Scl_N$ such that $\bar{\phi}(g) \geq 1$ and $D(\bar{\phi}) \leq 1$. We see that $|\Scl_N| = O(N)$ and the rate-constant only depends on the number of generators of the right-angled Artin group.
\item We obtain gap results even for elements which are not in the commutator subgroup. This suggests that it may be interesting to use Bavard's Dualtiy Theorem as a generalisation of stable commutator length to an invariant of general group elements $g \in G$. That is to study the supremum of $\bar{\phi}(g) / 2$ where $\bar{\phi}$ ranges over all homogeneous quasimorphisms with $D(\bar{\phi}) = 1$ which vanish or are bounded on a fixed generating set.
In \cite{calegari:ziggurats} the authors studied this supremum over all homogeneous quasimorphisms induced by circle actions. They could prove that this supremum has certain 
qualitative similarities to the experimental values observed for $\scl$. This includes the experimental phenomenon that values with low denominators appear more frequently in $\scl$.
\end{itemize}

\subsection{Organisation}
In Section \ref{sec:QM and Bavard} we introduce notation, definitions and basic or well established results on stable commutator length, quasimorphisms and Bavard's Duality Theorem.
In Section \ref{sec:letter-thin triples and alpha, beta} we introduce \emph{letter-thin triples} which are a special type of triples $(x_1,x_2,x_3)$ of alternating elements $x_1, x_2, x_3 \in \Acl$. 
These will be crucial in estimating the defect of the quasimorphisms constructed in this paper.
We will define maps $\alpha, \beta \col \Acl \to \Acl$, which we show to respect letter-thin triples in Lemma \ref{lemma:alpha keeps thin.}.
In Section \ref{sec:gaps via Letter-Quasimorphisms} we define and study \emph{letter-quasimorphisms} which are maps from arbitrary groups to alternating words of the free group. We deduce Theorem \ref{thm:main} which serves as a criterion for $\scl$-gaps of $1/2$ using these letter-quasimorphisms.
Section \ref{sec:Left orders and convex subgroups} recalls some results of \cite{convexsub} on left relatively convex subgroups and orders on groups.
Using the sign of these orders we are able to deduce $1/2$ gaps for amalgamated free products in Section \ref{sec:amalgamation}; see Theorem \ref{thm:amalgamation}. We show the $1/2$ gaps for right-angled Artin groups in Section \ref{sec:RAAGs and scl}; see Theorem \ref{thm:raags and scl}.

\subsection*{Acknowledgements} I would like to thank my supervisor, Martin Bridson, for his help, support and guidance, and Ric Wade for his very helpful comments. I would further like to thank the referee for carefully reading the paper and recommending helpful improvements.
Moreover, I would like to thank the Isaac Newton Institute for Mathematical Sciences in Cambridge for support and hospitality during the programme \emph{Non-Positive Curvature Group Actions and Cohomology} where work on this paper was undertaken. 
I would like to thank Danny Calegari for a stimulating conversation at the Isaac Newton Institute and Max Forester for pointing out errors in a previous version of this paper.
This work was supported by EPSRC grant no EP/K032208/1. The author is also supported by the Oxford-Cocker Scholarship.

\section{Quasimorphisms and Bavard's Duality Theorem} \label{sec:QM and Bavard}

In Subsection \ref{subsec:quasimorphisms and Bavard's Duality} we give basic properties and definitions of stable commutator length and Bavard's Duality Theorem. In Subsection \ref{subsec:spectral gaps of scl} we collect some known results on (spectral) gaps in stable commutator length. In Subsections \ref{subsec:generalised qm} we define generalised quasimorphisms and in Subsection \ref{subsec:brooks and 2-brooks} the well known Brooks quasimorphisms.

\subsection{Quasimorphisms and Bavard's Duality Theorem} \label{subsec:quasimorphisms and Bavard's Duality} 
For what follows Greek letters ($\alpha$, $\beta$) will denote generic functions, upper-case Roman letters ($A$, $B$) will denote generic groups, lower-case Roman letters ($a,b$) generic group elements and code-font ($\at$, $\bt$) will denote letters in a free group. We will stick to this notation unless it is mathematical convention to do otherwise.

Let $G$ be a group. For two elements $g,h \in G$ the \emph{commutator} is defined via $[g,h] = g h g^{-1} h^{-1}$
and the group generated by all such commutators is called the \emph{commutator subgroup} of $G$ and is denoted by $G'$.
For an element $g \in G'$ we set
\[
\cl(g) = \min \{ k \mid g = \prod_{i=1}^k [g_i,h_i]; g_i, h_i \in G \}
\]
the \emph{commutator length of $g$}. Note that $\cl$ is subadditive and hence the limit
\[
\scl(g) = \lim_{n \to \infty} \frac{\cl(g^n)}{n}
\]
exists and is called \emph{stable commutator length (scl)}. 
See \cite{calegari:scl} for a comprehensive reference on scl. Calegari showed that in non-abelian free groups scl can be computed efficiently in polynomial time and is rational. For a group $G$, the set of possible values of $\scl$ is not fully understood, even for non-abelian free groups. See Subsection \ref{subsec:spectral gaps of scl}
for a discussion on gaps in $\scl$.
We note the following basic property:
\begin{prop}
$\scl$ is monotone and characteristic. That is, for any group homomorphism $\theta \col G \to H$ and any $g \in G$ we have 
$\scl(g) \geq \scl( \theta(g))$. If $\theta$ is an automorphism, then $\scl(g) = \scl( \theta(g))$.
\end{prop}

A \emph{quasimorphism} is a map $\phi \col G \to \R$ such that there is a constant $D$, such that for all $g, h \in G$, $|\phi(g) + \phi(h) - \phi(gh) | \leq D$. The infimum of all such  $D$ is called the \emph{defect} of $\phi$ and denoted by $D(\phi)$. Note that quasimorphisms form a vector space under pointwise addition and scalar multiplication. 
A quasimorphism $\bar{\phi}$ is said to be \emph{homogeneous} if $\bar{\phi}(g^n) = n \bar{\phi}(g)$ for all $n \in \Z$, $g \in G$. In particular, $\bar{\phi}$ is \emph{alternating}, i.e. $\bar{\phi}(g^{-1}) = - \bar{\phi}(g)$ for all $g \in G$.

Every quasimorphism $\phi \col G \to \R$ is boundedly close to a unique homogeneous quasimorphism $\bar{\phi} \col G \to \R$  defined via
\[
\bar{\phi}(g) := \lim_{n \to \infty} \frac{\phi(g^n)}{n}
\]
and we call $\bar{\phi}$ the \emph{homogenisation} of $\phi$. Homogeneous quasimorphisms on $G$ form a vector space, denoted by $\mathcal{Q}(G)$.
\begin{prop} \label{prop:defect of homogenisation doubles}
Let $\phi \col G \to \R$ be a quasimorphism and let be $\bar{\phi}$ be its homogenisation. Then $D(\bar{\phi}) \leq 2 D(\phi)$.
\end{prop}
See Lemma 2.58 of \cite{calegari:scl} for a proof. For what follows we will \emph{always} decorate homogeneous quasimorphisms with a bar-symbol, even if they are not explicitly induced by a non-homogeneous quasimorphism. 
We refer the reader to \cite{frigerio} and \cite{calegari:scl} for references on quasimorphisms and stable commutator length.

If $g_1$ and $g_2$ lie in the same conjugacy class of $G$ then $\bar{\phi}(g_1) = \bar{\phi}(g_2)$, hence homogeneous quasimorphisms are class functions.
The key ingredient to calculate gaps in stable commutator length is Bavard's Duality Theorem:
\begin{thm} \label{thm:Bavards duality}
\cite{bavard} Let $G$ be a group and let $g \in G'$. Then
\[
\scl(g) = \sup_{\bar{\phi} \in \mathcal{Q}(G)} \frac{|\bar{\phi}(g)|}{2 D(\bar{\phi})}.
\]
\end{thm}
See \cite{calegari:scl} for a proof and a generalisation of this statement.
This theorem allows us to estimate stable commutator length using (homogeneous) quasimorphisms. It can be shown that 
the supremum in Bavard's Duality Theorem is obtained. That is, for every element $g \in G'$ there is a homogeneous quasimorphism  $\bar{\phi}$ with $D(\bar{\phi}) = 1$ such that $\scl(g) = \bar{\phi}(g)/2$. These quasimorphisms are called \emph{extremal} and were studied in \cite{calegari:extremal}.

\subsection{(Spectral) Gaps in scl} \label{subsec:spectral gaps of scl}

It was shown by \cite{DH} that every non-trivial element $w \in \F_n'$ in the commutator subgroup of a non-abelian free group satisfies that $\scl(w) \geq 1/2$ and that every non-trivial commutator $[w_1,w_2] \in \F_n$ satisfies $\scl([w_1,w_2]) = 1/2$.

Using the monotonicity of scl we may conclude that for an arbitrary group $G$ every commutator $[g_1,g_2] \in G'$ satisfies $\scl([g_1,g_2]) \leq 1/2$. On the other hand, some elements $g \in G'$ satisfy $\scl(g) = 0$ for trivial reasons, for example if they are torsions or a positive power of this element is conjugate to a negative power of this element.

We call the infimum of $\{ scl(g) > 0 \mid  g \in G' \}$ the \emph{gap of $\scl$}, often called the \emph{spectral gap}, and say that a group \emph{has a gap in scl} if this number is positive. Many classes of ``negatively curved'' groups have a gap in scl.
\begin{itemize}
\item Residually free groups have a gap of exactly $1/2$ by Duncan and Howie \cite{DH}.
\item Mapping class groups of closed orientable surfaces, possibly with punctures, have a gap depending on the surface; see \cite{BBF}.
\item Hyperbolic groups have a gap which depends on the hyperbolicity constant and the number of generators; see \cite{calegari_fujiwara}.
\item Some classes of groups may not have a uniform gap but the first accumulation point on conjugatcy classes of positive $\scl$ may be uniformly bounded away from zero. For example for non-elementary, torsion-free hyperbolic groups and for the fundamental groups of closed hyperbolic manifolds this accumulation point is at least $1/12$; see Theorem B of \cite{calegari_fujiwara} and see Theorem 3.11 of \cite{calegari:scl}.
\item Sometimes, one may control $\scl$ on certain generic group elements. If $G = G_1 \star G_2$ is the free product of two torsion-free groups $G_1$ and $G_2$ and $g \in G'$ does not conjugate into one of the factors, then $\scl(g) \geq 1/2$; see \cite{chen} and \cite{Ivanov-Klyachko}. Similarly, if $G = A \star_C B$ and $g \in G'$ does not conjugate into one of the factors and such that $C g C$ does not contain a copy of any conjugate of $g^{-1}$ then $\scl(g) \geq 1/12$. See Theorem D of \cite{calegari_fujiwara} for the first proof of this gap and \cite{scl_in_bs_groups} for the sharp gap and a generalisation to graphs of groups.
\item Baumslag--Solitar groups have a sharp uniform gap of $1/12$; see \cite{scl_in_bs_groups}.
\end{itemize}
Note that this list is not meant to be comprehensive. By monotinicity, having a gap in scl may serve as an obstruction for group embeddings. If $H$ and $G$ are non-abelian groups with $H \inj G$ and $C > 0$ is such that every non-trivial element $g \in G'$ satisfies $\scl(g) \geq C$ then so does every non-trivial element of $H'$.

\subsection{Generalised Quasimorphisms} \label{subsec:generalised qm}
It is possible to generalise quasimorphisms $\phi \col G \to \R$ to maps $\Phi \col G \to H$ for $G,H$ \emph{arbitrary groups}. Two quite different proposals for such a generalisation come from Fujiwara--Kapovich (\cite{fuji-kapo}) and Hartnick--Schweitzer (\cite{hartnick-schweitzer}). Whereas the former maps are quite restrictive, the latter type of maps are very rich.
The ``letter-quasimorphisms'' defined in this paper will be quasimorphisms as defined by Hartnick--Schweitzer as shown at the end of Subsection \ref{subsec:letter-quasimorphisms and well-behaved letter quasimorphisms}.
Adapting the definition of \cite{hartnick-schweitzer} we call a map $\Phi \col G \to H$ between arbitrary groups a \emph{quasimorphism} if for every (ordinary) quasimorphism $\alpha \col H \to \R$,  $\alpha \circ \Phi \col G \to \R$, i.e. the pullback of $\alpha$ to $G$ via $\Phi$, defines a quasimorphism on $G$.
A map $\phi \col G \to \R$ is a quasimorphism in the sense of Hartnick--Schweitzer if and only if it is an ordinary quasimorphism.

The quasimorphisms $G \to \R$ constructed in this paper will be all pullbacks of the most basic quasimorphisms $\F_2 \to \R$ via letter-quasimorphisms $G \to \Acl \subset \F_2$; see Remark \ref{rmk:quasimorphisms are pullback of hs qm}.

\subsection{Brooks Quasimorphisms} \label{subsec:brooks and 2-brooks}
For what follows $\F_2$ will denote the group on two generators $\at$ and $\bt$.
A word $w = \xt_1 \cdots \xt_k \in F(\{ \at, \bt \} ) = \F_2$ is called \emph{reduced} if it has no backtracking. 
Unless stated otherwise \emph{we will always assume that elements in the free group are represented by reduced words}.
A sub-letter $\xt_i$ is called a \emph{power of $\at$ (or $\bt$)} if 
$\xt_i \in \{ \at, \at^{-1} \}$ (or $\xt_i \in \{ \bt, \bt^{-1} \}$).
Furthermore, $w$ is called \emph{alternating} if the letters of $w$ alternate between an element in $\{ \at, \at^{-1} \}$ and an element in $\{ \bt, \bt^{-1} \}$. The set of alternating words of $\F_2 = \langle \at, \bt \rangle$ is denoted by $\mathcal{A}$.
A word $v = \yt_1 \cdots \yt_l$ is called \emph{subword} of $w = \xt_1 \cdots \xt_k$ if $l \leq k$ and there is an $n \in \{0, \ldots, k-l \}$ such that $\yt_i = \xt_{i+n}$ for every $i \in \{ 1, \ldots, l \}$.  

Let $w \in \F_2$, $g \in \F_2$ be arbitrary reduced words. Let $\nu_w(g)$ 
be the number of (possibly overlapping) subwords of $w$ in the reduced word $g$. Then the function
\[
\eta_w = \nu_w - \nu_{w^{-1}}
\]
is a quasimorphism, called \emph{Brooks quasimorphism}. These maps were introduced by Brooks in \cite{brooks} to show that the vector space of
(homogeneous) quasimorphisms of the free group is infinite dimensional. Observe that for a letter $\xt$, the map $\eta_\xt$ is a homomorphism. Brooks quasimorphisms have been vastly generalised to other cases and groups; see \cite{EF} and \cite{me}.

Let $g,h \in \F_2$ and let $(c_1, c_2, c_3)$ be reduced words such that $g = c_1^{-1} c_2$, $h= c_2^{-1} c_3$, $h^{-1} g^{-1} = c_3^{-1} c_1$ are reduced words. 
Then it is easy to see that the value $\eta_w(g) + \eta_w(h) - \eta_w(gh)$ only depends on the first $|w|-1$ letters of the words $c_1, c_2, c_3$, hence the defect is indeed finite.
There is an extremal Brooks quasimorphism to the basic commutator $[ \at, \bt ]$, namely $\eta_{\at \bt} - \eta_{ \bt \at }$. This and homomorphisms will be the only Brooks quasimorphisms occurring in this paper.

\begin{exmp} \label{exmp: extemal brooks quasimorphisms on free group}
Consider $[\at, \bt]$, the commutator of the letters $\at$ and $\bt$. Then it is easy to see that the quasimorphism $\eta_0 = \eta_{\at \bt} - \eta_{\bt \at}$ satisfies that $\eta_0([\at,\bt])=\bar{\eta_0}([\at,\bt])=2$, $D(\eta_0)=1$ and $D(\bar{\eta}_0) = 2$. As usual, $\bar{\eta}_0$ denotes the homogenisation of $\eta_0$. 
By Bavard's Duality Theorem (\ref{thm:Bavards duality}) we may estimate $\scl([\at,\bt ]) \geq \bar{\eta}([\at, \bt])/2 D(\bar{\eta}) = 1/2$ and, as $\scl([\at,\bt]) \leq 1/2$ (see Subsection \ref{subsec:spectral gaps of scl}), we conclude $\scl([\at,\bt])=1/2$ and see that $\bar{\eta}_0$ is extremal.
\end{exmp}

\subsection{Bounded Cohomology}
We define (bounded) cohomology of discrete groups and state its basic properties.
We refer the reader to \cite{frigerio} for a thorough treatment of the bounded cohomology of discrete groups.

Let $G$ be a group, let $V$ be a $\Z G$-module and set $C^n(G, V) = \{ f \col G^n \rightarrow V \}$. For what follows, $V = \Z$ or $V = \R$
and we think of $V$ as a $\Z G$-module with trivial action.
Let $\| \cdot \|_\infty$ be the $l^\infty$-norm on $C^n(G, \mathbb{R})$ and set 
\[
C^n_b(G, V) = \{ f \in C^n(G,V) \mid \|f \|_\infty < \infty \} \subset C^n(G,V)
\]
Define the well-known coboundary maps for the inhomogeneous resolution $\delta^n \col C^n(G, V) \rightarrow C^{n+1}(G, V)$ via
\begin{equ*}{rl}
 \delta^n(f) (g_1, \ldots, g_{n+1}) = &f(g_2, \ldots, g_{n+1}) + \sum_{i=1}^n (-1)^i f(g_1, \ldots, g_i g_{i+1}, \ldots, g_{n+1}) + \cdots \\
  &\cdots (-1)^{n+1} f(g_1, \ldots, g_n)
\end{equ*}
and note that $\delta^n$ restricts to $\delta^n \col C^n_b(G,V) \to C^{n+1}_b(G,V)$.
Set
\begin{align*}
Z^n_{(b)}(G,V) &= \textrm{ker} \big( \delta^n \col C_{(b)}^n(G, V) \to C_{(b)}^{n+1}(G,V) \big) \\
 B^n_{(b)}(G,V) &= \textrm{im} \left( \delta^{n-1} \col C_{(b)}^{n-1}(G, V) \to C_{(b)}^n(G,V) \right)
\end{align*}
the (bounded) cocycles $Z^n_{(b)}(G,V)$ and the (bounded) coboundaries $B^n_{(b)}(G,V)$.
Then $\mathrm{H}^n(G,V) = Z^n(G,V) / B^n(G,V)$ is called the \emph{ordinary cohomology} and $\mathrm{H}_b^n(G,V) = Z_b^n(G,V) / B_b^n(G,V)$ is called the \emph{bounded cohomology} of $G$ with coefficients in $V$. The embedding $C_b^n(G, \R) \inj C^n_(G, \R)$ induces a map $c^n \col \mathrm{H}_b^n(G,V) \to \mathrm{H}^n(G,V)$ called the \emph{comparison map}.

Let $\phi \col G \to \R$ be a quasimorphism. Then $\delta^1 \phi \in C^2_b(G,\R)$ is a bounded $2$-cocycle and hence induces a class $[\delta^1 \phi] \in H^2_b(G, \R)$.
These classes are exactly the classes which lie in the kernel of the comparison map $c^2 \col \mathrm{H}_b^2(G,\R) \to \mathrm{H}^2(G,\R)$ described above.

(Bounded) Cohomology is functorial in both slots:
Any homomorphism $\alpha \col G \to H$ induces a well defined map $\alpha^* \col \Hrm_{(b)}^n(H,V) \to \Hrm_{(b)}^n(G,V)$ on (bounded) cohomology by pulling back cocycles on $H$ to cocycles on $G$ via $\alpha$.
The map $\Z \to \R$ induces a \emph{change of coefficients} map $\Hrm_{(b)}^n(G,\Z) \to \Hrm_{(b)}^n(G,\R)$.

\subsection{Bounded Cocycles via Actions on the Circle and Vice Versa}

This subsection states a classical correspondence between bounded cohomology and circle actions developed by Ghys; see \cite{ghys}.
Also, see \cite{cicle_quasimorph_modern} for a thorough treatment of this topic.
Let $\Homeorm^+(S^1)$ be the group of orientation preserving actions on the circle and let
\[
\Homeorm^+_\Z(\R) = \{ f \in \Homeorm^+(\R) \mid \forall n \in \Z, x \in \R: f(x+n) = f(x)+n \}
\]
the subgroup of the orientation preserving homeomorphisms of the real line that commutes with the integers.
By identifying $S^1 \cong \R / \Z$ we obtain a surjection
$\pi \col \Homeorm^+_\Z(\R)  \to \Homeorm^+(S^1)$. The kernel of $\pi$ is isomorphic to $\Z$ via 
$\iota \col n \mapsto f_n$ with $f_n \col x \mapsto x+n$ and lies in the center of $\Homeorm^+_\Z(\R)$.
Hence 
\[
\begin{tikzcd}
 0 \arrow[r] & \Z \arrow[r,"\iota"] & \Homeorm^+_\Z(\R)  \arrow[r,"\pi"] & \Homeorm^+(S^1) \arrow[r] \arrow[l, bend right, "\sigma"] & 1
\end{tikzcd}
\]
is a central extension and hence corresponds to a class $\eurm \in\Hrm^2(\Homeorm^+(S^1), \Z)$ the \emph{Euler-class}.
This class is represented by the cocycle $\omega \col (g,h) \mapsto \sigma(g) \sigma(h) \sigma(gh)^{-1} \in \Z$ by identifying $\Z$ with $\mathrm{ker}(\pi)=\mathrm{im}(\iota)$ and where $\sigma$ is any set-theoretic section $\sigma \col \Homeorm^+(S^1)  \to \Homeorm^+_\Z(\R)$. Let $\sigma_b$ be the unique section such that $\sigma_b(f)(0) \in [0,1)$. Then $\omega_b(g,h) = \sigma_b(g) \sigma_b(h) \sigma_b(gh)^{-1}$ satisfies that for all $g,h \in G$, $\omega_b(g,h) \in \{ 0,1 \}$ and hence is $\omega_b$ is a \emph{bounded} cocycle. We call the class $\eurm_b = [\omega_b] \in \Hrm^2_b(\Homeorm^+(S^1), \Z)$ the \emph{bounded Euler class}. See \cite{me_extensions} for the correspondence of group extensions and bounded cohomology.
The image of $\eurm_b$ under the change of coefficients $\Hrm^2_b(\mathrm{Homeo}^+(S_1), \Z) \to \Hrm^2_b(\mathrm{Homeo}^+(S_1), \R)$
is called the \emph{real bounded Euler class} and denoted by $\eurm^\R_b$.

Any action $\rho \col G \to \Homeorm^+(S^1)$ induces a  bounded class via $\rho^*\eurm_b \in \Hrm^2_b(G,\Z)$ (resp. $\rho^*\eurm_b^\R \in \Hrm^2_b(G, \R)$).
Ghys (\cite{ghys}) showed that two actions $\rho_1, \rho_2 \col G \to \mathrm{Homeo}^+(S_1)$ 
are \emph{semi-conjugate} if and only if $\rho_1^*\eurm_b=\rho_2^*\eurm_b \in \Hrm^2_b(G,\Z)$.
See \cite{cicle_quasimorph_modern}  for a precise definition of semi-conjugacy.
Similarly, we have $\rho^* \eurm^\R_b = 0 \in \Hrm^2_b(G,\R)$ if and only if $\rho$ is semi-conjugate to an action by rotations.

The class $\rho^*\eurm_b \in \Hrm^2_b(G, \Z)$
may be represented by a cocycle $\rho^*\omega_b \in Z^2_b(G, \Z)$ such that for 
every $g,h \in G$, $\rho^*\omega_b(g,h) \in \{0,1 \}$.
Surprisingly, a converse statement holds:
\begin{thm} \label{thm:ghys}
\footnote{See \cite{ghys}, see also Theorem 1.3 of \cite{cicle_quasimorph_modern}} Let $G$ be a discrete countable group and let $[\omega] \in H^2_b(G,\Z)$ be a class represented by a cocycle $\omega$, such that for all $g,h \in G$, $\omega(g,h) \in \{ 0, 1 \}$. Then there is an action $\rho \col G \to \mathrm{Homeo}^+(S^1)$ such that
$\rho^*\eurm_b = [\omega] \in \mathrm{H}_b^2(G, \Z)$.
\end{thm}
 This allows us to show that certain quasimorphisms are induced by a circle action $\rho \col G \to \Homeorm^+(S^1)$ without explicitly constructing $\rho$.

\section{Letter-Thin Triples and the Maps $\alpha$ and $\beta$} \label{sec:letter-thin triples and alpha, beta}
The set of alternating words $\Acl \subset \F_2$ is the set of all words in the letters $\at$ and $\bt$ where the letters alternate between $\{ \at, \at^{-1} \}$ and $\{ \bt, \bt^{-1} \}$.
For example, $\at \bt \at^{-1} \bt^{-1}$ is an alternating word but $\at \bt \bt \at^{-1} \bt^{-1} \bt^{-1}$ is not.
We will define maps $\alpha, \beta \col \Acl \to \Acl$ and develop their basic properties in Subsection \ref{subsec:alpha and beta}.
We also define a version of these maps on $\bar{\Acl}_0$, which are conjugatcy classes of \emph{even}-length words of $\Acl$ to understand how $\alpha, \beta$ behave on powers; see Proposition \ref{prop:powers of alpha, beta}.
In Subsection \ref{subsec:letter-thin and alpha and beta} we define certain triples $(x_1,x_2,x_3)$ where $x_1,x_2,x_3 \in \Acl$ called \emph{letter-thin triples}. We think of them as the sides of (thin) triangles; see Figure \ref{fig:triangles}.
Note that such triples are not triangles in the usual sense, i.e. the sides $x_1,x_2,x_3$ do \emph{not} correspond to the geodesics between three points in some metric space like a Cayley graph. Letter-thin triples will be crucial in estimating the defect of the quasimorphisms we construct in this paper.
We will see that $\alpha$ and $\beta$ map letter-thin triples to letter-thin triples in Lemma \ref{lemma:alpha keeps thin.}, which is the main technical result of this paper.
In Subsection \ref{subsec:brooks qm, homomorphisms and letter-thin} we see that basic Brooks quasimorphisms and homomorphisms behave well on letter-thin triples.
We usually prove the properties we state for $\alpha, \beta$ just for $\alpha$ and note that all properties may be deduced analogously for $\beta$ by interchanging $\at$ and $\bt$; see Proposition \ref{prop:cutting words}, (\ref{prop-case:interchange a b}).

\subsection{The Maps $\alpha$ and $\beta$, Definition and Properties} \label{subsec:alpha and beta}
We will describe two maps $\alpha, \beta \col \mathcal{A} \to \mathcal{A}$ sending alternating words to alternating words.
Define $\Scl_\at^+, \Scl_\at^- \subset \Acl$ as 
\begin{align*}
\Scl_\at^+ & = \{ \at \yt_1 \at \cdots \at \yt_l \at \mid \yt_i \in \{ \bt, \bt^{-1} \}, l \in \N \} \\
\Scl_\at^- & = \{ \at^{-1} \yt_1 \at^{-1} \cdots \at^{-1} \yt_l \at^{-1} \mid \yt_i \in \{ \bt, \bt^{-1} \}, l \in \N \}
\end{align*}
that is, $\Scl_\at^+$ is the set of alternating words which start and end in $\at$ and don't contain the letter $\at^{-1}$ and $\Scl_\at^-$ is the set of alternating words which start and end in $\at^{-1}$ and don't contain the letter $\at$.
Note that we assume $0 \in \N$, i.e. $\at \in \Scl_\at^+$ and $\at^{-1} \in \Scl_\at^-$.

Analogously we define the sets $\Scl_\bt^+ \subset \Acl$ and  $\Scl_\bt^- \subset \Acl$ as
\begin{align*}
\Scl_\bt^+ & = \{ \bt \xt_1 \bt \cdots \bt \xt_l \bt \mid \xt_i \in \{ \at, \at^{-1} \}, l \in \N \} \\
\Scl_\bt^- & = \{ \bt^{-1} \xt_1 \bt^{-1} \cdots \bt^{-1} \xt_l \bt^{-1} \mid \xt_i \in \{ \at, \at^{-1} \}, l \in \N \}
\end{align*}
and observe that $\bt \in \Scl_\bt^+$ and $\bt^{-1} \in \Scl_\bt^-$.

We will decompose arbitrary words $w \in \Acl$ as a \emph{unique} product of elements in $\{ \bt, \bt^{-1} \}$ and $\Scl_\at^+ \cup \Scl_\at^-$:
\begin{prop} \label{prop: a decomposition}
Let $w \in \Acl$ be an alternating word. Then there are $\yt_0, \ldots, \yt_l$ and $s_1, \ldots, s_l$ such that
\[
w = \yt_0 s_1 \yt_1 s_2 \cdots \yt_{l-1} s_l \yt_l
\]
where $\yt_i \in \{\bt, \bt^{-1} \}$ except that $\yt_0$ and/or $\yt_l$ may be empty and $s_i \in \Scl_\at^+ \cup \Scl_\at^-$. Moreover, $s_i$ alternates between $\Scl_\at^+$ and $\Scl_\at^-$, i.e. there is no $i \in \{ 1, \ldots, l-1 \}$ such that $s_i, s_{i+1} \in \Scl_\at^+$  or 
$s_i, s_{i+1} \in \Scl_\at^-$.
This expression is unique.
\end{prop}
We will call this way of writing $w$ the \emph{ $\at$-decomposition of $w$}.
Analogously, we may also write $w \in \Acl$ as
\[
w = \xt_0 t_1 \xt_1 t_2 \cdots \xt_{l-1} t_l \xt_l
\] 
(possibly with a different $l$), where $\xt_i \in \{ \at, \at^{-1} \}$ except that $\xt_0$ and / or $\xt_l$ may be empty and $t_i \in \Scl_\bt^+ \cup \Scl_\bt^-$ where 
$t_i$ alternate between $\Scl_\bt^+$ and $\Scl_\bt^-$. We will call this way of writing $w$ the \emph{$\bt$-decomposition of $w$}.

\begin{proof} (of Proposition \ref{prop: a decomposition})
Let $w \in \Acl$ be an alternating word. 
Since $\at \in \Scl_\at^+$ and $\at^{-1} \in \Scl_\at^-$, we may always find some $s_i \in \Scl_\at^+ \cup \Scl_\at^-$ and some $\yt_i \in \{ \bt, \bt^{-1} \}$ such that
\[
w = \yt_0 s_1 \yt_1 s_2 \cdots \yt_{n-1} s_n \yt_n
\]
with possibly $\yt_n$ and / or $\yt_0$ empty.

Now let $m$ be the minimal $n$ of all such products representing $w$ i.e.
\[
w = \yt_0 s_1 \yt_1 s_2 \cdots \yt_{m-1} s_m \yt_m.
\]
Suppose there is an $i \in \{ 1, \ldots, m-1 \}$ such that $s_i,s_{i+1} \in \Scl_\at^+$ (resp. $s_i,s_{i+1} \in \Scl_\at^-$).
Set $s' = s_i \yt_i s_{i+1}$ and note that $s' \in \Scl_\at^+$ (resp. $s' \in \Scl_\at^-$).
Then
\[
w = \yt_0 s_1 \yt_1 s_2 \cdots \yt_{i-1} s' \yt_{i+1} \cdots \yt_{m-1} s_m \yt_m
\]
which would contradict the minimality of $m$. Hence all $s_i$ alternate between $\Scl_\at^+$ and $\Scl_\at^-$. 
By comparing two such expressions we see that such an expression is further unique.
\end{proof}

\begin{defn} \label{defn:alpha and beta}
Let $w \in \Acl$ and let $w = \yt_0 s_1 \cdots \yt_{l-1} s_l \yt_{l}$ be the $\at$-decomposition of $w$. Then $\alpha \col \Acl \to \Acl$ is defined via
\[
\alpha \col w \mapsto
\yt_0 \xt_1 \yt_1 \xt_2 \cdots \yt_{l-1} \xt_l \yt_{l}
\]
with $\xt_i = \at$ if $s_i \in \Scl^+_\at$ and $\xt_i = \at^{-1}$ if $s_i \in \Scl^-_{\at}$.

Analogously suppose that $w = \xt_0 t_1 \xt_1 t_2 \cdots \xt_{l-1} t_l \xt_l$ is the $\bt$-decomposition of $w$,
with $l$ possibly different from above. We define the map $\beta \col \Acl \to \Acl$ via
\[
 \beta \col w \mapsto \xt_0 \yt_1 \xt_1 \yt_2 \cdots \xt_{l-1} \yt_l \xt_{l}
\]
with $\yt_i = \bt$ if $t_i \in \Scl^+_\bt$ and $\yt_i = \bt^{-1}$ if $t_i \in \Scl^-_{\bt}$.
\end{defn}

\begin{exmp}
Let $w = \bt \at \bt^{-1} \at \bt \at \bt^{-1} \at^{-1} \bt \at^{-1} \bt \at \bt \at^{-1}$.
Then the $\at$-decomposition of $w$ is
\[
w = \bt s_1 \bt^{-1} s_2 \bt s_3 \bt s_4
\] 
where $s_1 = \at \bt^{-1} \at \bt \at \in \Scl^+_\at$, $s_2 =  \at^{-1} \bt \at^{-1} \in \Scl_\at^-$, $s_3 = \at \in \Scl^+_\at$ and $s_4 = \at^{-1} \in \Scl^+_\at$.
Hence
\[
\alpha(w) = \bt \at \bt^{-1} \at^{-1} \bt \at \bt \at^{-1}.
\]
Observe that then $\alpha(\alpha(w)) = \alpha(w)$.
The $\bt$-decomposition of $\alpha(w)$ is
\[
\alpha(w) = t_1 \at t_2 \at^{-1} t_3 \at^{-1}
\]
where $t_1 = \bt \in \Scl_\bt^+$, $t_2 = \bt^{-1} \in \Scl_\bt^-$ and $t_3 = \bt \at \bt \in \Scl_\bt^+$. 
Hence
\[
\beta(\alpha(w)) = \bt \at \bt^{-1} \at^{-1} \bt \at^{-1}
\]
and similarly, we may see that $\alpha(\beta(\alpha(w))) = \bt \at \bt^{-1} \at^{-1} = [\bt, \at]$. Then both
$\alpha([\bt, \at]) = [\bt, \at]$  and $\beta([\bt, \at])= [\bt,\at]$.
We will formalise and use this behaviour later; see Proposition \ref{prop:cutting words} and Proposition \ref{prop:alpha on conjugacy classes decreases}.
\end{exmp}

The images of $\alpha$ and $\beta$ are obviously contained in the set of alternating words. Moreover, as the $s_i$ in the previous definition all alternate between $\Scl^+_{\at}$ and $\Scl^-_{\at}$, none of the consecutive $\xt_i$ have the same sign in the image of $\alpha$ and no consecutive $\yt_i$ have the same sign in the image of $\beta$.

\begin{prop} \label{prop:cutting words} 
The maps $\alpha, \beta \col \Acl \to \Acl$ have the following properties:
\begin{enumerate}
\item \label{prop-case:alpha alternating} For every $w \in \Acl$, $\alpha(w^{-1}) = \alpha(w)^{-1}$ and $\beta(w^{-1}) = \beta(w)^{-1}$
\item \label{prop-case:interchange a b} $\psi \circ \alpha = \beta \circ \psi$ and $\psi \circ \beta = \alpha \circ \psi$, where $\psi \col \F_2 \to \F_2$ is the automorphism defined via $\psi \col \at \mapsto \bt, \bt \mapsto \at$.
\item  For any $w \in \Acl$, $\alpha(\alpha(w)) = \alpha(w)$. Moreover, $|\alpha(w)| \leq |w|$ with equality if and only if $\alpha(w) = w$. The analogous statement holds for $\beta$.
\item \label{prop:cases,splitting} Let $v_1 \xt v_2$ be an alternating word with $v_1, v_2 \in \Acl$ and $\xt \in \{ \at, \at^{-1} \}$. Then $\alpha(v_1 \xt v_2)$ is equal in $\F_2$ to the element represented by the non-reduced word
$\alpha(v_1 \xt) \xt^{-1} \alpha(\xt v_2)$. The analogous statement holds for $\beta$.
\end{enumerate}
\end{prop}

\begin{proof}
To see $(1)$, note that if $w = \yt_0 s_1 \yt_1 \cdots \yt_{l-1} s_l \yt_{l}$ is the $\at$-decomposition of $w$, then
\[
\yt_l^{-1} s_l^{-1} \yt_{l-1}^{-1}  \cdots \yt_1^{-1} s_1^{-1} \yt_0^{-1} 
\]
is the $\at$-decomposition of $w^{-1}$. As $s_i^{-1} \in \Scl_\at^+$ if and only if $s_i \in \Scl_\at^-$ and $s_i^{-1} \in \Scl_\at^-$ if and only if $s_i \in \Scl_\at^+$ we can conclude that $\alpha(w^{-1}) = \alpha(w)^{-1}$.
The analogous argument holds for $\beta$.

Point $(2)$ is evident from the symmetric way $\alpha$ and $\beta$ have been defined. To see $(3)$, note that $\alpha$ replaces each of the subwords $s_i$ by letters $\at$ or $\at^{-1}$. These
have size strictly less than $|s_i|$ unless $s_i$ is the letter $\at$ or $\at^{-1}$ already. 
This shows $|\alpha(w)| \leq |w|$ with equality only if $\alpha(w) = w$ and it also shows that $\alpha \circ \alpha = \alpha$.

For (\ref{prop:cases,splitting}) suppose that the $\at$-decomposition of $v_1 \xt$ is $\yt^1_0 s^1_1 \yt^1_1 \cdots \yt^1_{l_1-1} s^1_{l_1}$ and the
$\at$-decomposition of $\xt v_2$ is $s^2_1 \yt^2_1 \cdots \yt^1_{l_1-1} s^2_{l_2} \yt^2_{l_2}$. Both, $s^1_{l_1}$ and $s^2_1$ lie in the same set $S_\at^+$ or $S_\at^-$ depending if $\xt = \at$ or $\xt = \at^{-1}$. Without loss of generality assume that $\xt = \at$. The $\at$-decomposition of $v_1 \xt v_2$ may be seen to be $\yt^1_0 s^1_1 \yt^1_1 \cdots \yt^1_{l_1-1} s \yt^2_1 \cdots \yt^2_{l_2-1} s^2_{l_2} \yt^2_{l_2}$
where $s \in S_\at^+$ is equal to $s^1_{l_1} \at^{-1} s^2_1$ in $\F_2$.
Hence $\alpha(v_1 \at) = \yt^1_0 \xt^1_1 \yt^1_1 \cdots \yt^1_{l_1-1} \at$, $\alpha(\at v_2) = \at \yt^2_1 \cdots \yt^2_{l_2-1} \xt^2_{l_2} \yt^2_{l_2}$ and 
\[
\alpha(v_1 \xt v_2) = \yt^1_0 \xt^1_1 \yt^1_1 \cdots \yt^1_{l_1-1} \at \yt^2_1 \cdots \yt^2_{l_2-1} \xt^2_{l_2} \yt^2_{l_2}.
\]
Comparing terms finishes the proposition.
\end{proof}

To study how the maps $\alpha, \beta \col \Acl \to \Acl$ behave on powers of elements we need to define a version of them on conjugacy classes.
Let $\bar{\mathcal{A}}_0$ be the set conjugacy classes of even length alternating words.
Note that then necessarily every two representatives $w_1,w_2 \in \Acl$ of the same conjugacy class in $\bar{\Acl}_0$ are equal up to cyclic permutation of the letters. This is, there are elements $v_1, v_2 \in \Acl$ such that
$w_1 = v_1 v_2$ and $w_2 = v_2 v_1$ as reduced words.
Hence every representative $v \in \Acl$ of an element in $\bar{\Acl}_0$ is automatically reduced.
\begin{rmk} \label{rmk:on conjugacy classes for acl}
Every reduced representative $w \in \Acl$ of a class in $\bar{\Acl}_0$ has the same length.
Every homogeneous quasimorphism $\bar{\phi} \col \F_2 \to \R$ depends only on conjugacy classes and hence induces a well-defined map $\bar{\phi} \col \bar{\Acl}_0 \to \R$.
We say that an element $[w] \in \bar{\Acl}_0$ \emph{lies in the commutator subgroup} if one (and hence any) representative $w$ of $[w]$ lies in the commutator subgroup of $\F_2$.
\end{rmk}

\begin{defn} \label{defn:maps alpha bar and beta bar}
Define the map $\bar{\alpha} \col \bar{\mathcal{A}}_0 \to \bar{\mathcal{A}}_0$ as follows:
Let $[w] \in \bar{\Acl}_0$. If $[w] = e$ set $\bar{\alpha}([w]) = e$. Else choose a representative $w \in \Acl$  of $[w]$ that starts with a power of $\at$ and, as $w$ has even length, ends in a power of $\bt$. Suppose that $w$ starts with the letter $\xt \in \{ \at, \at^{-1} \}$ and write $w= \xt w'$ for $w' \in \Acl$ such that $\xt w'$ is reduced.  
Then define $\bar{\alpha} \col \bar{\mathcal{A}}_0 \to \bar{\mathcal{A}}_0$ via
\[
\bar{\alpha} \col [w] \mapsto [\alpha(\xt w' \xt) \xt^{-1}] \in \bar{\Acl}_0.
\]
Define $\bar{\beta} \col \bar{\mathcal{A}}_0 \to \bar{\mathcal{A}}_0$ analogously: For every element $[w] \in \bar{\Acl}_0$ choose a representative $w \in \Acl$ which starts with the letter $\yt \in \{ \bt, \bt^{-1} \}$ and write $w = \yt w'$. Then define $\bar{\beta} \col \bar{\Acl}_0 \to \bar{\Acl}_0$ via
\[
\bar{\beta} \col [w] \mapsto [\beta(\yt w' \yt ) \yt^{-1} ] \in \bar{\Acl}_0.
\] 
\end{defn}

To see that $\bar{\alpha}, \bar{\beta} \col \bar{\Acl}_0 \to \bar{\Acl}_0$ are well-defined, suppose that $w_1, w_2 \in \Acl$ are both even alternating words which start in a power of $\at$ and both represent the same element $[w_1] = [w_2] \in \bar{\Acl}_0$. Let $\xt_1, \xt_2 \in \{ \at, \at^{-1} \}$ be the first letters of $w_1$ and $w_2$.
Then there are elements $v_1, v_2 \in \Acl$ such that $w_1 = \xt_1 v_1 \xt_2 v_2$ as a reduced word and
$w_2 = \xt_2 v_2 \xt_1 v_1$. Then, by $(3)$ of Proposition \ref{prop:cutting words}, 
\begin{align*}
\alpha(w_1 \xt_1) \xt_1^{-1} = \alpha(\xt_1 v_1 \xt_2 v_2 \xt_1) \xt_1^{-1} &=  \alpha(\xt_1 v_1 \xt_2) \xt_2^{-1} \alpha(\xt_2 v_2 \xt_1) \xt_1^{-1} \\
\alpha(w_2 \xt_2) \xt_2^{-1} =\alpha(\xt_2 v_2 \xt_1 v_1 \xt_2) \xt_2^{-1} &=  \alpha(\xt_2 v_1 \xt_1) \xt_1^{-1} \alpha(\xt_1 v_1 \xt_2) \xt_2^{-1}
\end{align*}
which are conjugate in $\F_2$ and so 
$[\alpha(w_1 \xt_1) \xt_1^{-1}] = [\alpha(w_2 \xt_2) \xt_2^{-1}]$. 
This shows that $\bar{\alpha}$ is well defined and analogously that $\bar{\beta}$ is well defined.

The definition of $\bar{\alpha}$ given above is useful for performing calculations.
 However, there is a more geometric way to think about $\bar{\alpha}$ and $\bar{\beta}$ analogous to the definition of $\alpha$ and $\beta$.
A common way to depict conjugacy classes in the free group is via labels on a circle: Let $w = \zt_1 \cdots \zt_n \in \F_2$ be a cyclically reduced word in the letters $\zt_i$. Then $w$ labels a circle by cyclically labelling the sides of the circle counterclockwise by $\zt_1, \zt_2, \ldots, \zt_n$ so that $\zt_n$ is next to $\zt_1$ on the circle.
Two cyclically reduced words $w \in \F_2$ then yield the same labelling up to rotation if and only if they define the same conjugacy class.

Let $[w] \in \bar{\Acl}_0$ be a conjugacy class of a word $w \in \Acl$ of even length that contains both at least one $\at$ and one $\at^{-1}$ as a subword. We may similarly define an $\at$-decomposition of such a cyclic labelling. 
One may show that in this geometric model the maps $\bar{\alpha}$ (resp. $\bar{\beta}$) can then be defined just like for $\alpha$ and $\beta$ by replacing the words in $\Scl_\at^+$ by $\at$ and the words in $\Scl_\at^-$ by $\at^-$.
If $[w] \in \bar{\Acl}_0$ does not contain both $\at$ and $\at^{-1}$ as subwords then $\bar{\alpha}([w])=e$ in both cases.
\begin{figure} 
  \centering

 \subfloat[]{\includegraphics[width=0.6\textwidth]{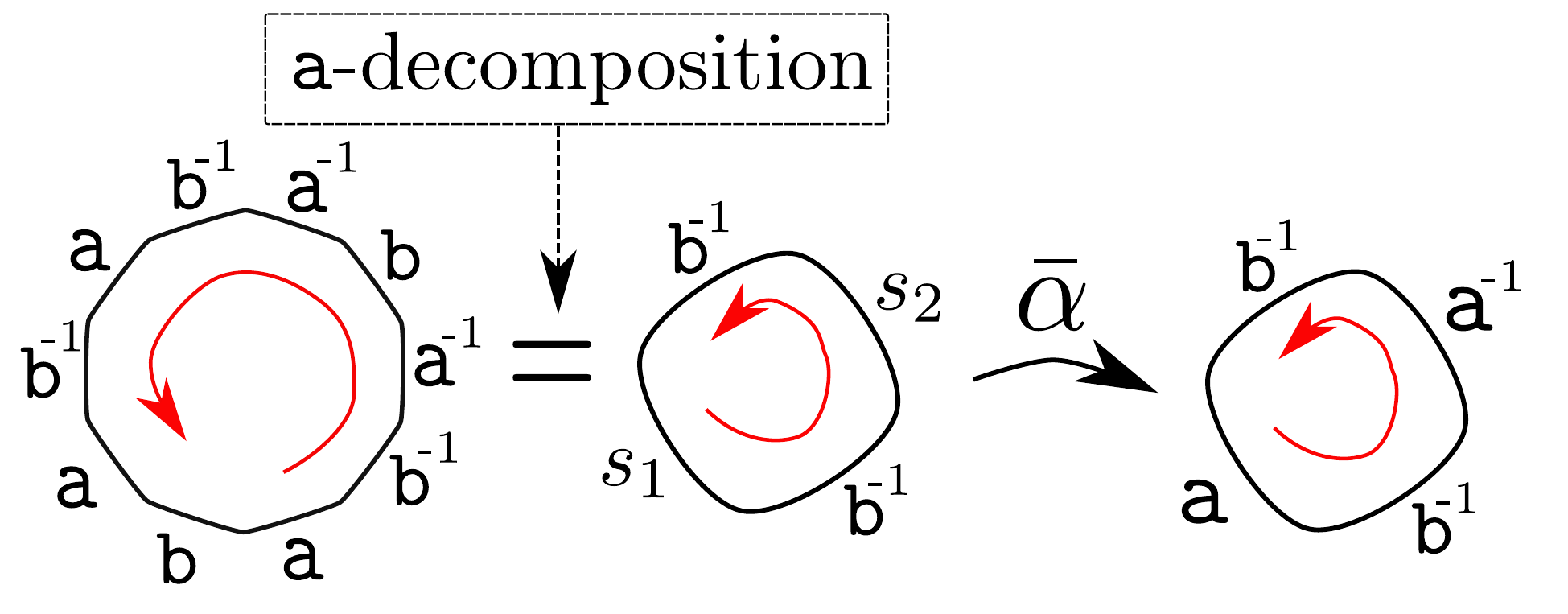} \label{fig:circle general}}
 \\
  \subfloat[]{\includegraphics[width=0.3\textwidth]{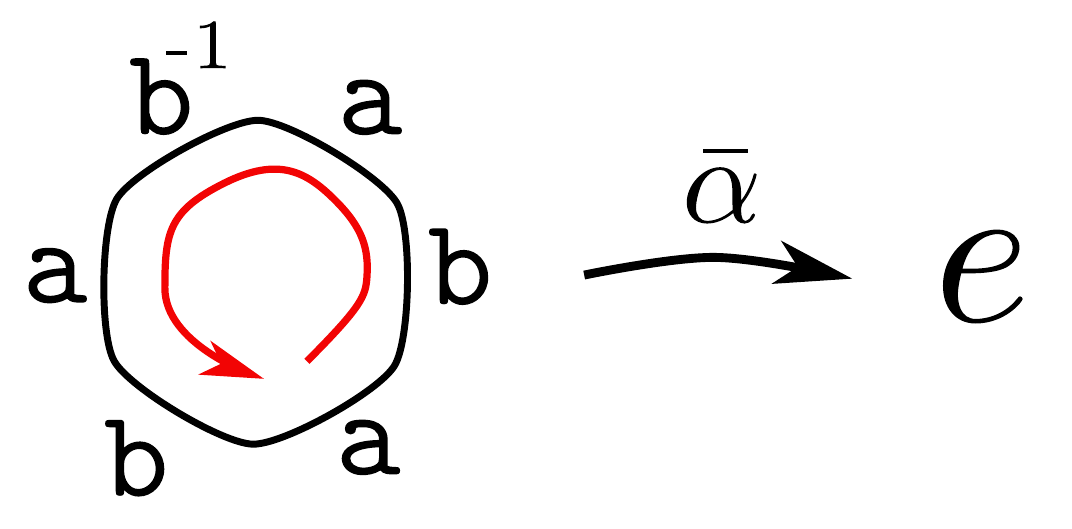} \label{fig:circle trivial}}

  \caption{Visulizing $\bar{\alpha}$: Conjugacy classes $[w]$ correspond to cyclic labels of a circle. One may define a $\at$-decomposition and $\bar{\alpha}$ on such labels except when $[w]$ does not contain $\at$ or $\at^{-1}$ as a subword. See Example \ref{exmp:alpha bar}}
  \label{fig:circles and conjugacy classes}
\end{figure}
Consider the following example:
\begin{exmp} \label{exmp:alpha bar}
Let $w = \at \bt^{-1} \at^{-1} \bt \at^{-1} \bt^{-1} \at \bt^{-1} \at \bt \in \Acl$. Its conjugacy class is depicted in Figure 
\ref{fig:circles and conjugacy classes}. We observe that $w$ starts with $\at$ and set $w' = \bt^{-1} \at^{-1} \bt \at^{-1} \bt^{-1} \at \bt^{-1} \at \bt$ so that $w = \at w'$. By Definition \ref{defn:maps alpha bar and beta bar}, $\bar{\alpha}([w]) = [\alpha(\at w' \at) \at^{-1}] = [\left( \at \bt^{-1} \at^{-1} \bt^{-1} \at \right) \at^{-1} ] = [\at \bt^{-1} \at^{-1} \bt^{-1}]$. However, we could have also done an $\at$-decomposition of the elements on a circle as pictured in Figure
\ref{fig:circles and conjugacy classes} ($\textrm{A}$) with $s_1 = \at \bt^{-1} \at \bt \at \in \Scl_\at^+$ and $s_2 = \at^{-1} \bt \at^{-1} \in \Scl_\at^-$ and obtained the same result.

Similarly, let $w = \at \bt \at \bt^{-1} \at \bt$. It's conjugacy class is represented by a cyclic labelling of a circle in Figure \ref{fig:circles and conjugacy classes} (\textrm{B}). The first letter of $w$ is $\at$. Set $w' = \bt \at \bt^{-1} \at \bt$ so that
$w = \at w'$.
The $\at$-decomposition of $\at w' \at = s_1 \in \Scl_{\at}^+$. Hence $\bar{\alpha}([w]) = [\alpha(\at w' \at) \at^{-1}] = [ \left( \at \right) \at^{-1}] = [e] \in \bar{\Acl}_0$.
\end{exmp}

\begin{prop} \label{prop:alpha on conjugacy classes decreases}
Let $\bar{\alpha}, \bar{\beta} \col \bar{\Acl}_0 \to \bar{\Acl}_0$ be defined as above and let $[w] \in \bar{\Acl}_0$. Then $|\bar{\alpha}([w])| \leq |[w]|$ with equality if and only if $\bar{\alpha}([w]) = [w]$. The analogous statement holds for $\bar{\beta}$.
If $[w]$ is a non-trivial class in the commutator subgroup of $\F_2$ then $\bar{\alpha}([w])$ and $\bar{\beta}([w])$ are non-trivial.
If $\bar{\alpha}([w]) = [w] = \bar{\beta}([w])$ then $[w]$ may be represented by $w = [\at, \bt]^n$ for $n \in \Z$.
\end{prop}

\begin{proof}
To see that $\bar{\alpha}$, $\bar{\beta}$ decrease length unless they fix classes is the same argument as in the proof of Proposition \ref{prop:cutting words}.
If $[w]$ is a non-trivial class in the commutator subgroup of $\F_2$ then there is a reduced representative $w$ such that $w = \at v_1 \at^{-1} v_2$ for some appropriate $v_1,v_2 \in \Acl$ and we see that $\bar{\alpha}([w])$ is non-trivial as it also contains the subletters $\at$ and $\at^{-1}$. 
If $w \in \Acl$ is a representative such that $\bar{\alpha}$ fixes $[w]$ then $w$ has to be of the form
$w = \prod_{i=1}^k \at \yt_i \at^{-1} \yt'_i$ for some $\yt_i,\yt'_i \in \{ \bt, \bt^{-1} \}$, $k \geq 1$ and similarly, if $\bar{\beta}$ fixes a class then the a representative has to be of the form $w = \prod_{i=1}^k \xt_i \bt \xt'_i \bt^{-1}$ for some $\xt_i, \xt'_i \in \{ \at, \at^{-1} \}$, $k \geq 1$. Comparing both yields the statement.
\end{proof}

\begin{prop} \label{prop:powers of alpha, beta}
Assume that $w \in \Acl$ is
non-empty, has even length and that $c_1, c_2 \in \Acl$ are words such that $c_1 w c_2 \in \Acl$ is again an alternating word.
 Then there are words $d_1, d_2, w' \in \Acl$ 
such that $\alpha(c_1 w^n c_2) = d_1 w'^{n-1} d_2 \in \Acl$ for all $n \geq 1$ as reduced words where $w'$ has even length and $[w'] = \bar{\alpha}([w]) \in \bar{\Acl}_0$. If $w$ lies in the commutator subgroup then $w'$ is non-empty. The analogous statement holds for $\beta$.
\end{prop}

\begin{proof}
If $w \in \Acl$ does not contain both a positive and a negative power of $\at$, the statement follows by an easy calculation. Note that this is the case if and only if $\bar{\alpha}([w]) = [e]$.
Otherwise $w$ contains at least one sub-letter $\at$ and one sub-letter $\at^{-1}$.
This is the case if $w$ lies in the commutator subgroup. Suppose without loss of generality that $w = v_1 \at v_2 \at^{-1} v_3$ as a reduced word for some $v_1, v_2, v_3 \in \Acl$. By multiple applications of Proposition \ref{prop:cutting words}, we see that
\begin{align*}
\alpha(c_1 w^n c_2) &= \alpha(c_2 \left( v_1 \at v_2 \at^{-1} v_3 \right)^n c_2)  \\
&= \alpha(c_1 v_1 \at) \at^{-1} \alpha(\at v_2 \at^{-1} v_3 (v_1 \at v_2 \at^{-1} v_3)^{n-1} c_2) \\
&= \alpha(c_1 v_1 \at) \at^{-1} \alpha(\at v_2 \at^{-1} v_3 v_1 \at) \at^{-1} \alpha( \at v_2 \at^{-1} v_3 (v_1 \at v_2 \at^{-1} v_3)^{n-2} c_2) \\
&= \alpha(c_1 v_1 \at) \at^{-1} \left( \alpha(\at v_2 \at^{-1} v_3 v_1 \at) \at^{-1} \right)^2 \alpha( \at v_2 \at^{-1} v_3 (v_1 \at v_2 \at^{-1} v_3)^{n-3} c_2) \\
 &= \cdots \\
&= \alpha(c_1 v_1 \at) \at^{-1} (\alpha(\at v_2 \at^{-1} v_3 v_1 \at) \at^{-1} )^{n-1} \alpha(\at v_2 \at^{-1} v_3 c_2)
\end{align*} 
as non-reduced elements in the free group.
Then we define $d_1$, $d_2$ and $w'$ to be the reduced representative of 
\[
\alpha(c_1 v_1 \at) \at^{-1}, \text{   } \alpha(\at v_2 \at^{-1} v_3 c_2) \text{ and } \alpha(\at v_2 \at^{-1} v_3 v_1 \at) \at^{-1} 
\]
respectively. 
Moreover, $\alpha(\at v_2 \at^{-1} v_3 v_1 \at)$ is a reduced alternating word which starts and ends in $\at$ and contains the $\at^{-1}$ as a sub-letter. If follows that $w'$, the reduced representative of 
$\alpha(\at v_2 \at^{-1} v_3 v_1 \at) \at^{-1}$, starts with $\at$, contains $\at^{-1}$ and ends with a power of $\bt$, so $w'$ is non-empty. Further observe that $\bar{\alpha}([\at v_2 \at^{-1} v_3 v_1])$ is represented by $\alpha(\at v_2 \at^{-1} v_3 v_1 \at) \at^{-1}$ and hence $[w'] = \bar{\alpha}(w)$.
\end{proof}

\subsection{Letter-Thin Triples, $\alpha$ and $\beta$} \label{subsec:letter-thin and alpha and beta}
In order to streamline proofs later and ease notation we define an equivalence relation on triples $(x_1,x_2,x_3)$.
We think of such a triple as the sides of a (thin) triangle.
We stress that the $x_i$ are not actually the side of triangles in some metric space;
see Figure \ref{fig:triangles}.
Here, we study a special type of triples, namely \emph{letter-thin triples} in Definition \ref{defn:letter-thin}.

\begin{defn} \label{defn:equivalent triples}
Let $(x_1,x_2,x_3)$ be a triple of elements in $\F_2$ and let $\phi \col \F_2 \to \F_2$ be a set-theoretic function. We will understand by $\phi(x_1,x_2,x_3)$ the triple  $(\phi(x_1), \phi(x_2), \phi(x_3))$.
We define $\sim$ to be the equivalence relation on triples generated by
\begin{enumerate}
\item[(i)]  $(x_1,x_2,x_3) \sim (x_2, x_3, x_1)$
\item[(ii)] $(x_1,x_2,x_3) \sim (x_3^{-1}, x_2^{-1}, x_1^{-1})$
\item[(iii)] $(x_1,x_2,x_3) \sim \phi_\at(x_1, x_2, x_3)$, where $\phi_\at \col \F_2 \to \F_2$ is the automorphism defined via $\at \mapsto \at^{-1}$ and $\bt \mapsto \bt$.
\item[(iv)] $(x_1,x_2,x_3) \sim \phi_\bt(x_1, x_2, x_3)$, where $\phi_\bt \col \F_2 \to \F_2$ is the automorphism defined via $\at \mapsto \at$ and $\bt \mapsto \bt^{-1}$.
\end{enumerate}
for all $x_1,x_2,x_3 \in \F_2$ and say that $(x_1,x_2,x_3)$ is \emph{equivalent} to $(y_1, y_2, y_3)$ if $(x_1,x_2,x_3) \sim (y_1, y_2, y_3)$ under this relation.
\end{defn}
Imagining $(x_1,x_2,x_3)$ as labelling the sides of a triangle, two triples are equivalent if they may be obtained from each other by a sequence of rotations $(i)$, flips $(ii)$ or by changing the signs of its labels $(iii)$ \& $(iv)$.

\begin{prop} \label{prop:alpha respects equivalence}
Let $x_1,x_2,x_3,y_1,y_2,y_3 \in \F_2$ such that $(x_1,x_2,x_3) \sim (y_1, y_2, y_3)$.
Then if $x_1,x_2,x_3 \in \Acl$ also $y_1,y_2,y_3 \in \Acl$.
Moreover, in this case
$\alpha(x_1,x_2,x_3) \sim \alpha(y_1,y_2,y_3)$ and $\beta(x_1,x_2,x_3) \sim \beta(y_1, y_2, y_3)$.
\end{prop}

\begin{proof}
The first part is clear from the definitions.
Note that $\alpha$ commutes both with ``rotating the side''  $(i)$ and taking inverses $(ii)$ as $\alpha$ satisfies that $\alpha(w^{-1}) = \alpha(w)^{-1}$ for $w \in \Acl$.

Let $w = \yt_0 s_1 \yt_1 \cdots \yt_{k-1} s_k \yt_k$ be the $\at$-decomposition of $w$ (see Definition \ref{defn:alpha and beta}), where $\yt_i \in \{ \bt, \bt^{-1} \}$ and $s_i \in \Scl^+_\at \cup \Scl^-_\at$ alternates between 
$\Scl^+_\at$ and $\Scl^-_\at$.
Then 
\[
\phi_\at(w) = \yt_0 \phi_\at(s_1) \yt_1 \cdots \yt_{k-1} \phi_\at(s_k) \yt_k
\] where $\phi(s_i) \in \Scl^+_\at$ if and only if $s_i \in \Scl^-_\at$ and $\phi(s_i) \in \Scl^-_\at$ if and only if $s_i \in \Scl^+_\at$. So $\alpha(\phi_\at(w)) = \phi_\at(\alpha(w))$ and hence $\alpha \circ \phi_\at(x_1, x_2, x_3)$ is equivalent to $\alpha(x_1,x_2,x_3)$. Similarly, $\phi_\bt(w) = \phi_\bt(\yt_0)  \phi_\bt(s_1) \phi_\bt(\yt_1) \cdots \phi_\bt(\yt_{k-1}) \phi_\bt(s_k) \phi_\bt(\yt_k)$ where 
both $\phi_\bt(s_i)$ and $s_i$ lie in the same set $\Scl_\at^+$ or $\Scl_\at^-$. We see that once more, 
$\alpha(\phi_\bt(w)) = \phi_\bt(\alpha(w))$ and hence also $\alpha \circ \phi_\bt(x_1,x_2,x_3)$ is equivalent to $\alpha(x_1,x_2,x_3)$. Analogously, we see the statement for $\beta$.
\end{proof}

For a visualisation of the following definition we refer the reader to Figure \ref{fig:triangles}.
\begin{defn} \label{defn:letter-thin}
Let $x_1, x_2, x_3 \in \Acl$ be \emph{alternating} elements. The triple
$(x_1, x_2, x_3)$ is called \emph{letter-thin triple} in one of the following cases:
\begin{itemize}
\item[ $\To$ ] There are (possibly trivial) elements $c_1, c_2, c_3 \in \Acl$ such that
\begin{itemize}
\item[$\Toa$] $
(x_1,x_2,x_3) \sim (c_1^{-1} \at \bt c_2, c_2^{-1} \bt^{-1} \at c_3, c_3^{-1} \at^{-1} c_1)$ or
\item[$\Tob$]
$
(x_1,x_2,x_3) \sim (c_1^{-1} \bt \at c_2, c_2^{-1} \at^{-1} \bt c_3, c_3^{-1} \bt^{-1} c_1)$ 
\end{itemize}
where all words are required to be reduced.
\item[ $\Tt$ ] 
There are (possibly trivial) elements $c_1, c_2 \in \Acl$ such that
\begin{itemize}
\item[$\Tta$] 
$(x_1,x_2,x_3) \sim (c_1^{-1} \bt^{-1} \at \bt c_2, c_2^{-1} \bt^{-1}, \bt c_1)$ or
\item[$\Ttb$]
$(x_1,x_2,x_3) \sim (c_1^{-1} \at^{-1} \bt \at c_2, c_2^{-1} \at^{-1}, \at c_1)$
\end{itemize}
where all words are required to be reduced.
\end{itemize}
In all cases, $\sim$ denotes the equivalence of triples of Definition \ref{defn:equivalent triples}. We say that a letter-thin triple $(x_1,x_2,x_3)$ is \emph{of type $\Toa$, $\Tob$, $\Tta$} or \emph{$\Ttb$}
if it is equivalent to the corresponding triple above. 
\end{defn}
Note for example in the representatives of $\Toa$ above, necessarily $c_1$, $c_3$ are either empty or their first letter is a power of $\bt$. Similarly, $c_2$ is either empty or its first letter is a power of $\at$, else the $x_i$ would not be alternating.

Note that for any letter-thin triple $(x_1,x_2,x_3)$ of type $\Toa$ we may always find elements $d_1, d_2, d_3 \in \Acl$ with first letter a power of $\bt$ such that
\begin{align} \label{equ:maybe letter thin}
(x_1,x_2,x_3) = (d_1^{-1} \xt_1 d_2, d_2^{-1} \xt_2 d_3, d_3^{-1} \xt_3 d_1)
\end{align}
where $\xt_i \in \{ \at, \at^{-1} \}$ are such that \emph{not all of $\xt_1$, $\xt_2$ and $\xt_3$ are equal} i.e. have the same parity.
As we consider the triples only up to equivalence one may wonder if we can assume that any triple as in Equation (\ref{equ:maybe letter thin}) such that not all of $d_i$ are empty is letter-thin of type $\Toa$.
However, this is not the case: As $\xt_1$, $\xt_2$, $\xt_3$ do not all have the same parity, there is exactly one $i$ such that $\xt_i = \xt_{i+1}$ where indices are considered$\mod 3$. Then one may see that $(x_1,x_2,x_3)$ is of type $\Toa$ \emph{if and only if} $d_{i+1}$ is non-trivial. For example, $(d_1^{-1} \at, \at d_3, d_3^{-1} \at^{-1} d_1)$ is \emph{not} letter-thin for any $d_1, d_3 \in \Acl$ empty or starting with a power of $\bt$.

\begin{exmp}
$(\at, \at, \at^{-1})$ is not letter-thin and by the previous discussion also the triple $(\bt^{-1} \at^{-1} , \at^{-1} \bt, \bt^{-1} \at \bt)$ is not letter-thin. However, $(\bt^{-1} \at^{-1} \bt, \bt^{-1} \at^{-1}, \at \bt)$ \emph{is} letter-thin.
To see this, note that
\begin{equ*}{rcl}
(\bt^{-1} \at^{-1} \bt, \bt^{-1} \at^{-1}, \at \bt) &\overset{(iii)}{\sim}& (\bt^{-1} \at \bt, \bt^{-1} \at, \at^{-1} \bt) \\
& = & (c_1^{-1} \at \bt c_2, c_2^{-1} \bt^{-1} \at c_3, c_3^{-1} \at^{-1} c_1)
\end{equ*}
for $c_1 = \bt$, $c_2 = e$ and $c_3 = e$ and where $\overset{(iii)}{\sim}$ denotes the equivalence $(iii)$ of the definition of '$\sim$'; see Definition \ref{defn:equivalent triples}.
\end{exmp}

Note that by definition, if $(x_1,x_2,x_3)$ is letter-thin then \emph{all $x_1, x_2, x_3$ are alternating words}. 

See Figure \ref{fig:triangles} for the explanation of the name \emph{letter-thin triple}: First consider elements $g,h \in \F_2 = \langle \at, \bt \rangle$. The triple $(g,h,(gh)^{-1})$ corresponds to sides of a geodesic triangle in the Cayley graph $\textrm{Cay}(\F_2, \{ \at, \bt \})$ with endpoints $e, g, gh$. Note further that there are words $c_1, c_2, c_3 \in \F_2$ such that $g = c_1^{-1} c_2$, $h = c_2^{-1} c_3$, $(gh)^{-1} = c_3^{-1} c_1$ and all these expressions are freely reduced. A \emph{letter-thin} triple $(x_1,x_2,x_3)$ is such that each $x_i$ is in addition alternating and corresponds \emph{almost} to the sides of a geodesic triangle in a Cayley graph, apart from one letter $r \in \{ \at, \bt \}$ in the ``middle'' of the triangle. Figure \ref{fig:triangles} (\textrm{B}) corresponds to case $\To$ of Definition \ref{defn:letter-thin}, Figure \ref{fig:triangles} (\textrm{C}) corresponds to case $\Tt$ of Definition \ref{defn:letter-thin}. These letter-thin triples $(x_1,x_2,x_3)$ do \emph{not} label sides of triangles in a Cayley graph or any other metric space.
 \begin{figure}
  \centering

  \subfloat[]{\includegraphics[width=0.3\textwidth]{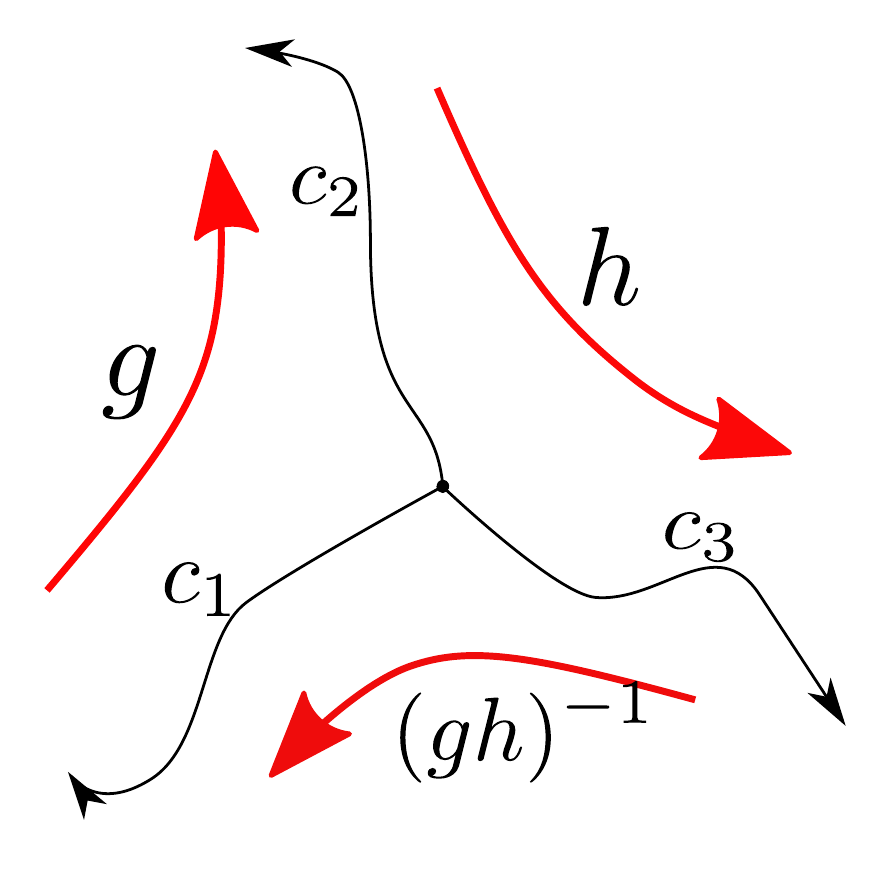} \label{fig:thin triangle}} 
  \hfill
  \subfloat[]{\includegraphics[width=0.3\textwidth]{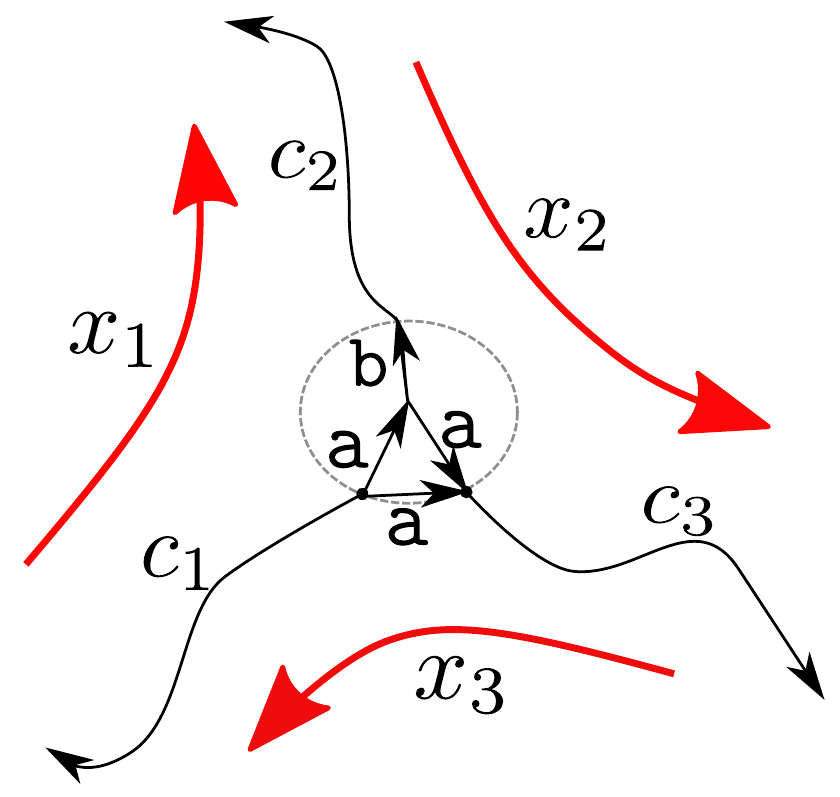} \label{fig:letter-thin triple}}
  \hfill
  \subfloat[]{\includegraphics[width=0.3\textwidth]{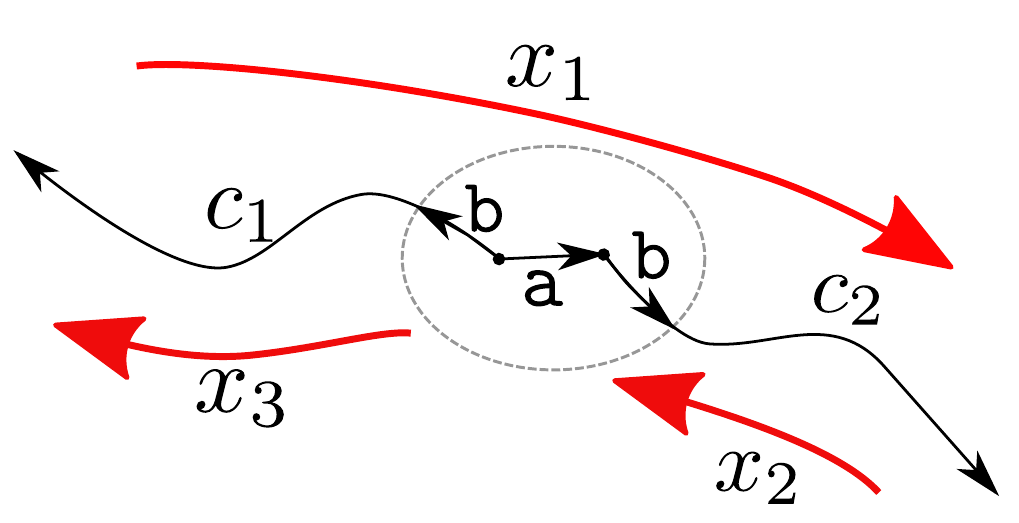} \label{fig:letter-thin degenerate}}
  
  \caption{Different ``triangles'': (\textrm{A}) arises as a generic thin triangle in the Cayley graph $\mathrm{Cay}(\F_2, \{ \at, \bt \})$ of the free group with standard generating set. Figures (\textrm{B}) and (\textrm{C}) correspond to letter-thin triples $\Toa$, $\Tta$. The grey dotted circles indicate the part of the letter-thin triples which can not be empty.
  These letter-thin triples do \emph{not} live in a Cayley graph or any well-known metric space.}. \label{fig:triangles}
\end{figure}

Observe that $(x_1,x_2,x_3)$ is letter-thin if and only if $\psi(x_1,x_2,x_3)$ is letter-thin for $\psi$ defined as in   Proposition \ref{prop:cutting words} (\ref{prop-case:interchange a b}) i.e. $\psi$ is the automorphism $\psi \col \F_2 \to \F_2$ defined via $\psi \col \at \mapsto \bt$ and $\psi \col \bt \mapsto \at$.

The maps $\alpha$ and $\beta$ respect letter-thin triples:
\begin{lemma} \label{lemma:alpha keeps thin.}
If $(x_1,x_2,x_3)$ is letter-thin. Then both $\alpha(x_1, x_2, x_3)$ and $\beta(x_1,x_2,x_3)$ are letter-thin. 
\end{lemma}

\begin{proof}
We will proceed as follows:
Let $(x_1,x_2,x_3)$ be a letter-thin triple.
By Proposition \ref{prop:alpha respects equivalence} it is enough to check that $\alpha(x_1,x_2,x_3)$ is letter-thin for one representative of the equivalence class. Hence it suffices to check that $\alpha(x_1, x_2, x_3)$ is letter thin for
\begin{enumerate}
\item Type $\Toa$: $(x_1,x_2,x_3) = (c_1^{-1} \at \bt c_2, c_2^{-1} \bt^{-1} \at c_3, c_3^{-1} \at^{-1} c_1)$
\item Type $\Tob$: $(x_1,x_2,x_3) = (c_1^{-1} \bt \at c_2, c_2^{-1} \at^{-1} \bt c_3, c_3^{-1} \bt^{-1} c_1)$
\item Type $\Tta$: $(x_1,x_2,x_3) = (c_1^{-1} \bt^{-1} \at \bt c_2, c_2^{-1} \bt^{-1}, \bt c_1)$
\item Type $\Ttb$: $(x_1,x_2,x_3) = (c_1^{-1} \at^{-1} \bt \at c_2, c_2^{-1} \at^{-1}, \at c_1)$
\end{enumerate}
By symmetry, this will show the analogous statement for $\beta$.

Proposition \ref{prop:cutting words},  (\ref{prop:cases,splitting}) allows us to compute $\alpha$ piecewise i.e. after each occurrence of a letter $\at$ or $\at^{-1}$ in a reduced word. For any reduced word $c \in \Acl$ starting with a power of $\bt$ or being empty, we will write $c_+$ for the reduced word represented by $\at^{-1} \alpha(\at c)$, which itself is not reduced since $\alpha(\at c)$ starts with an $\at$. Similarly, we will write $c_-$ for the reduced word represented by $\at \alpha(\at^{-1} c)$.
Note that $c_+$ and $c_-$ are either empty or their first letter is a power of $\bt$, as $\alpha(\at^{\pm} c)$ is alternating.
 If $c$ is a word which already has a subscript, say $c_i$, then we will write $c_{i,+}$ and $c_{i,-}$, respectively.
We consider each of the above cases independently. 
For letter-thin triples $(x_1,x_2,x_3)$ of type $\Toa$ we compute $\alpha(x_1,x_2,x_3)$
and we will state exactly which equivalences $(i)$, $(ii)$, $(iii)$ and $(iv)$ of Definition \ref{defn:equivalent triples} are needed to obtain one of the representatives for $\Toa$, $\Tob$, $\Tta$ and $\Ttb$ of letter-thin triples as in Definition \ref{defn:letter-thin}. For letter-thin triples $(x_1,x_2,x_3)$ of type $\Tob$, $\Tta$ and $\Ttb$ we will just state the type of $\alpha(x_1,x_2,x_3)$ without explicitly giving the equivalence.

\begin{enumerate}
\item Type $\Toa$: Suppose $(x_1,x_2,x_3) = (c_1^{-1} \at \bt c_2, c_2^{-1} \bt^{-1} \at c_3, c_3^{-1} \at^{-1} c_1)$.
As $(x_1,x_2,x_3)$ are alternating $c_2$ is either empty or starts with a positive or a negative power of $\at$.
We consider these cases separately:
\begin{itemize}
\item $c_2$ is empty. In this case 
we compute using Proposition \ref{prop:cutting words},
\begin{align*}
\alpha(c_1^{-1} \at \bt) &= \alpha(c_1^{-1} \at) \at^{-1} \alpha( \at \bt) = \alpha( \at^{-1} c_1)^{-1} \bt = (\at^{-1} c_{1,-})^{-1} \bt = (c_{1,-})^{-1} \at \bt \\
\alpha(\bt^{-1} \at c_3) &= \alpha(\bt^{-1} \at) \at^{-1} \alpha(\at c_3) = \bt^{-1} \at c_{3,+}   \\
\alpha(c_3^{-1} \at^{-1} c_1) &= \alpha(c_3^{-1} \at^{-1}) \at \alpha(\at^{-1} c_1) =
\alpha(\at c_3)^{-1} c_{1,-} = (\at c_{3,+})^{-1} c_{1,-} = (c_{3,+})^{-1} \at^{-1} c_{1,-}
\end{align*}
and hence
\[
\alpha(x_1, x_2, x_3)=((c_{1,-})^{-1} \att \bt, \bt^{-1} \att c_{3,+}, (c_{3,+})^{-1} \att^{-1} c_{1,-})
\] which is of type $\Toa$. Indeed, for 
$c_1' = c_{1,-}$, $c_2'=e$ and $c_3' = c_{3,+}$ we see that 
\[
\alpha(x_1, x_2, x_3) = ({c'_1}^{-1} \at \bt c'_2, {c'_2}^{-1} \bt^{-1} \at c'_3, {c'_3}^{-1} \at^{-1} c'_1).
\]
and hence $\alpha(x_1,x_2,x_3)$ is of type $\Toa$.
\item $c_2 = \at d_2$ where, $d_2 \in \Acl$.
\[
\alpha(x_1, x_2, x_3)=((c_{1,-})^{-1} \att d_{2,+}, (d_{2,+})^{-1} \at^{-1} \bt^{-1} \att c_{3,+}, (c_{3,+})^{-1} \att^{-1} c_{1,-})
\]
which is of type $\Ttb$ if $c_{1,-}$ is trivial and of type $\Tob$ else.
To see this we distinguish between three different cases:
\begin{itemize}
\item $c_{1,-}$ is trivial: Then 
\begin{align*}
\alpha(x_1, x_2, x_3) &=(\att d_{2,+}, (d_{2,+})^{-1} \at^{-1} \bt^{-1} \att c_{3,+}, (c_{3,+})^{-1} \att^{-1}) \\
 &\overset{(i)}{\sim} ((d_{2,+})^{-1} \at^{-1} \bt^{-1} \att c_{3,+}, (c_{3,+})^{-1} \att^{-1}, \att d_{2,+}) \\
 &\overset{(iv)}{\sim} (\phi_b(d_{2,+})^{-1} \at^{-1} \bt \att c_{3,+}, \phi_b(c_{3,+})^{-1} \att^{-1}, \att \phi_b(d_{2,+})) \\
 &= ({c'_1}^{-1} \at^{-1} \bt \at c'_2, {c'_2}^{-1} \at^{-1}, \at c'_1)
\end{align*}
for $c'_1 = \phi_b(d_{2,+})^{-1}$ and $c'_2 =  c_{3,+}$ and hence of type $\Ttb$. Here $\sim$ denotes the equivalences on triples defined in Definition \ref{defn:equivalent triples} with the corresponding numbering $(i) - (iv)$.
\item $c_{1,-}$ is non-trivial and starts with first letter $\bt$. Then define $d_1$ via $c_{1,-} = \bt d_1$.
Hence $\alpha(x_1, x_2, x_3)$ equals:
\begin{equ*}{ccl}
 & &(d_1^{-1} \bt^{-1} \att d_{2,+}, (d_{2,+})^{-1} \at^{-1} \bt^{-1} \att c_{3,+}, (c_{3,+})^{-1} \att^{-1} \bt d_1) \\
&\overset{(iv)}{\sim} &(\phi_b(d_1)^{-1} \bt \at \phi_b(d_{2,+}), \phi_b(d_{2,+})^{-1} \at^{-1} \bt \at \phi_b(c_{3,+}), \phi_b(c_{3,+})^{-1} \at^{-1} \bt^{-1} \phi_b(d_1)) \\
&= &({c'_1}^{-1} \bt \at c'_2, {c'_2}^{-1} \at^{-1} \bt c'_3, {c'_3}^{-1} \bt^{-1} c'_1)
\end{equ*}
for $c'_1 = \phi_b(d_1)$, $c'_2 = \phi_b(d_{2,+})$, $c'_3 = \at \phi_b(c_{3,+})$ and hence is of type $\Tob$.

\item $c_{1,-}$ is non-trivial and starts with first letter $\bt^{-1}$. Then define $d_1$ via $c_{1,-} = \bt^{-1} d_1$.
Hence $\alpha(x_1, x_2, x_3)$ equals:
\begin{equ*}{ccl}
 & &(d_1^{-1} \bt \att d_{2,+}, (d_{2,+})^{-1} \at^{-1} \bt^{-1} \att c_{3,+}, (c_{3,+})^{-1} \att^{-1} \bt^{-1} d_1) \\
 & \overset{(ii)}{\sim}  
& d_1^{-1} \bt \at c_{3,+}, (c_{3,+})^{-1} \at^{-1} \bt \at d_{2,+}, (d_{2,+})^{-1} \at^{-1} \bt^{-1} d_1) \\
 & = & 
 ({c'_1}^{-1} \bt \at c'_2, {c'_2}^{-1} \at^{-1} \bt c'_3, {c'_3}^{-1} \bt^{-1} c'_1)
\end{equ*}
for $c'_1 = d_1$, $c'_2 = c_{3,+}$, $c'_3 = \at d_{2,+}$ and hence of type $\Tob$.
\end{itemize}

\item $c_2 =  \at^{-1} d_2$ where $d_2 \in \Acl$.
\[
\alpha(x_1, x_2, x_3)=((c_{1,-})^{-1} \att \bt \at^{-1} d_{2,-}, (d_{2,-})^{-1}  \att c_{3,+}, (c_{3,+})^{-1} \att^{-1} c_{1,-})
\]
which is of type $\Tob$ if $c_{3,+}$ is non-trivial and of type $\Ttb$, else. This can be seen analogously to the previous case.
\end{itemize}

\item Type $\Tob$: Suppose $(x_1,x_2,x_3) = (c_1^{-1} \bt \at c_2, c_2^{-1} \at^{-1} \bt c_3, c_3^{-1} \bt^{-1} c_1)$.
Up to equivalence, there are the following sub-cases:
\begin{itemize}
\item Both of $c_1, c_3$ are empty. Then
\[
\alpha(x_1, x_2, x_3)= (  \bt \at c_{2,+}, (c_{2,+})^{-1} \at^{-1} \bt, \bt^{-1} )
\]
which is of type $\Tob$
\item $c_1$ is not empty, $c_3$ is empty.
Then either 
\begin{itemize} 
\item $c_1 = \at d_1$. In this case
\[
\alpha(x_1, x_2, x_3)= ((d_{1,+})^{-1} \at^{-1} \bt \at c_{2,+},(c_{2,+})^{-1} \at^{-1} \bt, \bt^{-1} \at d_{1,+})
\]
which is of type $\Tob$ 
\item $c_1 = \at^{-1} d_1$. In this case
\[
\alpha(x_1, x_2, x_3)= ((d_{1,-})^{-1} \at c_{2,+},(c_{2,+})^{-1} \at^{-1} \bt, \bt^{-1} \at^{-1} d_{1,+})
\]
which is of type $\Toa$.
\end{itemize}
\item $c_1$ is empty and $c_3$ is not. Then either 
\begin{itemize}
\item $c_3 = \at d_3$, in which case
\[
\alpha(x_1, x_2, x_3)= ( \bt \at c_{2,+},(c_{2,+})^{-1} \at^{-1} \bt \at d_{3,+}, (d_{3,+})^{-1} \at^{-1} \bt^{-1})
\]
which is of type $\Tob$. 
\item $c_3 = \at^{-1} d_3$, in which case
\[
\alpha(x_1, x_2, x_3)= ( \bt \at c_{2,+},(c_{2,+})^{-1} \at^{-1} d_{3,-}, (d_{3,-})^{-1} \at \bt^{-1})
\]
which is of type $\Toa$. 
\end{itemize}

\item Both of $c_1, c_3$ are non-empty. Then either 
\begin{itemize}
\item $c_1 = \at d_1$, $c_3 = \at d_3$. In this case
\[
\alpha(x_1, x_2, x_3)= ( (d_{1,+})^{-1} \at^{-1} \bt \at c_{2,+},(c_{2,+})^{-1} \at^{-1} \bt \at d_{3,+}, (d_{3,+})^{-1} \at^{-1} \bt^{-1} \at d_{1,+})
\]
which is of type $\Tob$.
\item $c_1 = \at d_1$, $c_3 = \at^{-1} d_3$. In this case
\[
\alpha(x_1, x_2, x_3)= ( (d_{1,+})^{-1} \at^{-1} \bt \at c_{2,+},(c_{2,+})^{-1} \at^{-1} d_{3,-}, (d_{3,-})^{-1} \at  d_{1,+})
\]
which is of type $\Tob$ if $d_{3,-}$ is non-trivial, and of type $\Ttb$, else.
\item $c_1 = \at^{-1} d_1$, $c_3 = \at d_3$. In this case
\[
\alpha(x_1, x_2, x_3)= ( (d_{1,-})^{-1} \at c_{2,+},(c_{2,+})^{-1} \at^{-1} \bt \at d_{3,+}, (d_{3,+})^{-1} \at^{-1} d_{1,-})
\]
which is of type $\Tob$ if $d_{1,-}$ is non-trivial and of type $\Ttb$, else.
\item $c_1 = \at^{-1} d_1$, $c_3 = \at^{-1} d_3$. In this case
\[
\alpha(x_1, x_2, x_3)= ( (d_{1,-})^{-1} \at c_{2,+},(c_{2,+})^{-1} \at^{-1} \ d_{3,-}, (d_{3,-})^{-1} \at \bt^{-1} \at^{-1} d_{1,-})
\]
which is of type $\Tob$ if $c_{2,+}$ is non-trivial and of type $\Ttb$, else.
\end{itemize}
\end{itemize}

\item Type $\Tta$: Suppose $(x_1,x_2,x_3) = (c_1^{-1} \bt^{-1} \at \bt c_2, c_2^{-1} \bt^{-1}, \bt c_1)$. 
We distinguish between the following cases
\begin{itemize}
\item Both of $c_1, c_2$ are empty. Then
\[
\alpha(x_1,x_2,x_3) = (\bt^{-1} \at \bt, \bt^{-1}, \bt)
\]
which is of type $\Tta$.
\item One of $c_1, c_2$ is empty. Up to equivalence and changing indices we may assume that $c_2$ is empty. Then either
\begin{itemize}
\item $c_1 = \at d_1$ in which case
\[
\alpha(x_1,x_2,x_3) = ((d_{1,+})^{-1} \at^{-1} \bt^{-1} \at \bt, \bt^{-1}, \bt \at d_{1,+})
\]
which is of type $\Tta$ or 
\item $c_1 = \at^{-1} d_1$ in which case
\[
\alpha(x_1,x_2,x_3) = ((d_{1,-})^{-1} \at \bt, \bt^{-1}, \bt \at^{-1} d_{1,-})
\]
which is of type $\Tob$.
\end{itemize}
\item Both of $c_1, c_2$ are non-empty. Then either
\begin{itemize}
\item $c_1 = \at d_1$, $c_2 = \at d_2$ in which case
\[
\alpha(x_1,x_2,x_3) = ((d_{1,+})^{-1} \at^{-1} \bt^{-1} \at d_{2,+}, (d_{2,+})^{-1} \at^{-1} \bt^{-1}, \bt \at d_{1,+})
\]
which is of type $\Tob$ or 
\item $c_1 = \at d_1$, $c_2 = \at^{-1} d_2$ in which case
\[
\alpha(x_1,x_2,x_3) = ((d_{1,+})^{-1} \at^{-1} \bt^{-1} \at \bt \at^{-1} d_{2,-}, (d_{2,-})^{-1} \at \bt^{-1}, \bt \at d_{1,+})
\] 
which is of type $\Tta$ or
\item $c_1 = \at^{-1} d_1$, $c_2 = \at d_2$ in which case
\[
\alpha(x_1,x_2,x_3) = ((d_{1,-})^{-1} \at d_{2,+}, (d_{2,+})^{-1} \at^{-1} \bt^{-1}, \bt \at^{-1} d_{1,-})
\]
which is of type $\Toa$ or 
\item $c_1 = \at^{-1} d_1$, $c_2 = \at^{-1} d_2$ in which case
\[
\alpha(x_1,x_2,x_3) = ((d_{1,-})^{-1} \at \bt \at^{-1} d_{2,-}, (d_{2,-})^{-1} \at \bt^{-1}, \bt \at^{-1} d_{1,-})
\]
which is of type $\Tob$.
\end{itemize}
\end{itemize}

\item Type $\Ttb$: Suppose $(x_1,x_2,x_3) = (c_1^{-1} \at^{-1} \bt \at c_2, c_2^{-1} \at^{-1}, \at c_1)$.
We see that
\[
\alpha(x_1,x_2,x_3) = ((c_{1,+})^{-1} \at^{-1} \bt \at c_{2,+}, (c_{2,+})^{-1} \at^{-1}, \at c_{1,+} )
\]
which is of type $\Ttb$.

\end{enumerate}
This concludes the proof of Lemma \ref{lemma:alpha keeps thin.}.
\end{proof}

\subsection{Brooks Quasimorphisms, Homomorphisms and Letter-Thin Triples} \label{subsec:brooks qm, homomorphisms and letter-thin}

For what follows we want to study how the Brooks quasimorphism
$\eta_0 = \eta_{\at \bt} - \eta_{\bt \at}$ defined in
Example \ref{exmp: extemal brooks quasimorphisms on free group} or certain homomorphisms behave
on letter-thin triples. This will be done in Propositions \ref{prop:letter thin triples and two quasimorphisms} and \ref{prop: letter thin triples and homomorphisms}, respectively. 

\begin{prop} \label{prop:letter thin triples and two quasimorphisms}
Let $\eta_0 = \eta_{\at \bt} - \eta_{\bt \at} \col \F_2 \to \Z$ be as above. Then
\[
|\eta_0(x_1) + \eta_0(x_2) + \eta_0(x_3)| = 1
\]
for every letter-thin triple $(x_1,x_2,x_3)$. In particular $\eta_0(x_1)+\eta_0(x_2)+\eta_0(x_3) \in \{ -1, +1 \}$.
\end{prop}

\begin{proof}
First note that if $w = w_1 w_2 \in \F_2$ as a reduced word and if $\zt_1$ is the last letter of $w_1$ and $\zt_2$ is the first letter of $w$, then 
\begin{align} \label{equ:split up words brooks}
\eta_0(w) &= \eta_0(w_1) + \eta_0(\zt_1 \zt_2) + \eta_0(w_2).
\end{align}
Let $(x_1,x_2,x_3)$ be a triple. Note that the value
\[
|\eta_0(x_1) + \eta_0(x_2) + \eta_0(x_3)|
\]
is invariant under the equivalences $(i)$ and $(ii)$ of Definition \ref{defn:equivalent triples}. Up to these equivalences we see that any letter-thin triple $(x_1,x_2,x_3)$ is equivalent via $(i)$ and $(ii)$ to the following:
\begin{itemize}
\item Type $\Toa$:  $(c_1^{-1} \xt \yt c_2, c_2^{-1} \yt^{-1} \xt c_3, c_3^{-1} \xt^{-1} c_1)$, for $\xt \in \{ \at, \at^{-1} \}$ and $\yt \in \{ \bt, \bt^{-1} \}$. 
If $c_i$ is empty set $\zt_i = e$. Else let $\zt_i$ be the first letter of $c_i$.
Then, by using successively Equation (\ref{equ:split up words brooks}) we see that
\begin{align*}
\eta_0(x_1)  &= \eta_0(c_1^{-1}) + \eta_0(\zt_1^{-1} \xt) + \eta_0(\xt \yt) + \eta_0(\yt \zt_2) + \eta_0(c_2) \\
\eta_0(x_2) &= \eta_0(c_2^{-1}) + \eta_0(\zt_2^{-1} \yt^{-1}) + \eta_0(\yt^{-1} \xt) + \eta_0(\xt \zt_3) + \eta_0(c_3) \\
\eta_0(x_3) &= \eta_0(c_3^{-1}) + \eta_0(\zt_3^{-1} \xt^{-1}) + \eta_0(\xt^{-1} \zt_1) + \eta_0(c_1)
\end{align*}
Using that $\eta_0(c^{-1}) = - \eta_0(c)$ for any $c \in \F_2$ we see that
\[
|\eta_0(x_1) + \eta_0(x_2) + \eta_0(x_3)| = |\eta_0(\xt \yt) + \eta_0(\yt^{-1} \xt)|
\]
and hence we see that for any choice $\xt \in \{ \at, \at^{-1} \}$, $\yt \in \{ \bt, \bt^{-1} \}$ 
\[
|\eta_0(x_1) + \eta_0(x_2) + \eta_0(x_3)| = 1.
\]
\item Type $\Tob$: $(c_1^{-1} \yt \xt c_2, c_2^{-1} \xt^{-1} \yt c_3, c_3^{-1} \yt^{-1} c_1)$, for $\xt \in \{ \at, \at^{-1} \}$ and $\yt \in \{ \bt, \bt^{-1} \}$. This case is analogous to the previous case.
\item Type $\Tta$: $(c_1^{-1} \yt^{-1} \xt \yt c_2, c_2^{-1} \yt^{-1}, \yt c_1)$, for $\xt \in \{ \at, \at^{-1} \}$ and $\yt \in \{ \bt, \bt^{-1} \}$. Again, if $c_i$ is empty set $\zt_i = e$. Else let $\zt_i$ be the first letter of $c_i$.
By successively using Equation (\ref{equ:split up words brooks}) we see that
\begin{align*}
\eta_0(x_1) &= \eta_0(c_1^{-1}) + \eta_0(\zt_1^{-1} \yt^{-1}) + \eta_0(\yt^{-1} \xt) + \eta_0(\xt \yt) + \eta_0(\yt \zt_2) + \eta_0(c_2) \\
\eta_0(x_2) &= \eta_0(c_2^{-1}) + \eta_0(\zt_2^{-1} \yt^{-1}) \\
\eta_0(x_3) &= \eta_0(\yt \zt_1) + \eta_0(c_1)
\end{align*}
and again we observe that
\begin{align*}
|\eta_0(x_1) + \eta_0(x_2) + \eta_0(x_3)| &= |\eta_0(\yt^{-1} \xt) + \eta_0(\xt \yt)| = 1
\end{align*}
for any choice of $\xt \in \{ \at, \at^{-1} \}$, $\yt \in \{ \bt, \bt^{-1} \}$.

\item Type $\Ttb$: $(c_1^{-1} \xt^{-1} \yt \xt \bt c_2, c_2^{-1} \xt^{-1}, \xt c_1)$, for $\xt \in \{ \at, \at^{-1} \}$ and $\yt \in \{ \bt, \bt^{-1} \}$. This case is analogous to the previous case.
\end{itemize}
\end{proof}

Recall that $\eta_\xt \col \F_2 \to \Z$ denotes the homomorphism which counts the letter $\xt$.
\begin{prop} \label{prop: letter thin triples and homomorphisms}
Let $\eta = \eta_\xt + \eta_\yt \col \F_2 \to \Z$ for 
$\xt \in \{ \at, \at^{-1} \}$ or $\yt \in \{ \bt, \bt^{-1} \}$.
Then
\[
|\eta(x_1) + \eta(x_2) + \eta(x_3)| = 1
\]
for any $(x_1,x_2,x_3)$ letter-thin. In particular $\eta(x_1) + \eta(x_2) + \eta(x_3) \in \{ -1, +1 \}$.
\end{prop}

\begin{proof}
Let $\eta$ be as in the proposition and suppose that $(x_1,x_2,x_3)$
is letter-thin.
Just like in the proof of the previous proposition we will consider the four different types of letter thin triples up to equivalences $(i)$ and $(ii)$ of Definition \ref{defn:equivalent triples}.

\begin{itemize}
\item Type $\Toa$:  $(c_1^{-1} \xt \yt c_2, c_2^{-1} \yt^{-1} \xt c_3, c_3^{-1} \xt^{-1} c_1)$, for $\xt \in \{ \at, \at^{-1} \}$ and $\yt \in \{ \bt, \bt^{-1} \}$. We directly calculate, using that $\eta$ is a homomorphism:
\begin{align*}
\eta(x_1) &= \eta(c_1^{-1} \xt \yt c_2) =  -\eta(c_1) + \eta(\xt) + \eta(\yt) + \eta(c_2) \\
\eta(x_2) &= \eta(c_2^{-1} \yt^{-1} \xt c_3) = -\eta(c_2) - \eta( \yt) + \eta(\xt) + \eta(c_3) \\
\eta(x_3) &= \eta(c_3^{-1} \xt^{-1} c_1) = - \eta(c_3) - \eta(\xt) + \eta(c_1)
\end{align*}
and hence
\[
|\eta(x_1)+ \eta(x_2) + \eta(x_3)| = |\eta(\xt)| = 1
\]
for any $\xt \in \{ \at, \at^{-1} \}$.
\item Type $\Tob$: $(c_1^{-1} \yt \xt c_2, c_2^{-1} \xt^{-1} \yt c_3, c_3^{-1} \yt^{-1} c_1)$, for $\xt \in \{ \at, \at^{-1} \}$ and $\yt \in \{ \bt, \bt^{-1} \}$. This case is analogous to the previous case.
\item Type $\Tta$: $(c_1^{-1} \yt^{-1} \xt \yt c_2, c_2^{-1} \yt^{-1}, \yt c_1)$, for $\xt \in \{ \at, \at^{-1} \}$ and $\yt \in \{ \bt, \bt^{-1} \}$. Again we calculate
\begin{align*}
\eta(x_1) &= \eta(c_1^{-1} \yt^{-1} \xt \yt c_2) = - \eta(c_1) - \eta(\yt) + \eta(\xt) + \eta(\yt) + \eta(c_2) \\
\eta(x_2) &= \eta(c_2^{-1} \yt^{-1}) = -\eta(c_2) - \eta(\yt) \\
\eta(x_3) &= \eta(\yt c_1) = \eta(\yt)+\eta(c_1)
\end{align*}
and hence again
\[
|\eta(x_1)+ \eta(x_2) + \eta(x_3)| = |\eta(\xt)| = 1
\]
for any $\xt \in \{ \at, \at^{-1} \}$.
\item Type $\Ttb$: $(c_1^{-1} \xt^{-1} \yt \xt \bt c_2, c_2^{-1} \xt^{-1}, \xt c_1)$, for $\xt \in \{ \at, \at^{-1} \}$ and $\yt \in \{ \bt, \bt^{-1} \}$. This case is analogous to the previous case.
\end{itemize}
\end{proof}

\section{Gaps via Letter-Quasimorphisms} \label{sec:gaps via Letter-Quasimorphisms}

The aim of this section is to define letter-quasimorphisms and deduce the criterion for $1/2$ gaps in $\scl$. There will be two types of letter-quasimorphisms: \emph{(general) letter-quasimorphisms} (Definition~\ref{defn:letter quasihomomorphism}) and \emph{well-behaved letter-quasimorphisms} (Definition \ref{defn:well behaved letter quasimorphisms}). The former is useful for application, the latter will be useful for proofs. 
For each letter-quasimorphism $\Phi \col G \to \Acl$ there will be an associated well-behaved letter-quasimorphism $\tilde{\Phi} \col G \to \Acl$
where $\tilde{\Phi}(g)$ is obtained from $\Phi(g)$ by modifying its beginning and its end; see Proposition \ref{prop:every letter-qm induces well behaved}. 

\subsection{Letter-Quasimorphisms and Well-Behaved Letter-Quasimorphisms} \label{subsec:letter-quasimorphisms and well-behaved letter quasimorphisms}
As always $\Acl$ denotes the set of alternating words of $\F_2$ in the generators $\at$ and $\bt$.
\begin{defn} \label{defn:letter quasihomomorphism}
Let $G$ be a group. We say that $\Phi \col G \to \mathcal{A}$ is a \emph{letter-quasimorphism} if $\Phi$ is alternating, i.e. $\Phi(g^{-1}) = \Phi(g)^{-1}$ for every $g \in G$ and if for every $g,h \in G$ one of the following holds:
\begin{enumerate}
\item \label{defn:letter-qm:thin} $\Phi(g) \Phi(h) \Phi(g h)^{-1} = e$, or
\item \label{defn:letter-qm:general} there are elements $c_1, c_2, c_3 \in \Acl$ and letters $\xt_1, \xt_2, \xt_3$ 
such that either $\xt_1, \xt_2, \xt_3 \in \{ \at, \at^{-1} \}$ and $\xt_1 \xt_2 \xt_3 \in \{ \at, \at^{-1} \}$ or
 $\xt_1, \xt_2, \xt_3 \in \{ \bt, \bt^{-1} \}$ and $\xt_1 \xt_2 \xt_3 \in \{ \bt, \bt^{-1} \}$ which satisfy
 that $\Phi(g) = c_1^{-1} \xt_1 c_2$, $\Phi(h) = c_2^{-1} \xt_2 c_3$ and $\Phi(g h)^{-1} = c_3^{-1} \xt_3 c_1$ as freely reduced alternating words.
\end{enumerate}
\end{defn}

The motivating example for letter-quasimorphisms is the following:
\begin{exmp} \label{exmp:letter quasimorphisms on free group}
Consider the map $\Phi \col \F_2 \to \Acl$ defined as follows. Suppose that $w \in \F_2$ 
has reduced representation $\at^{n_1} \bt^{m_1} \cdots \at^{n_k} \bt^{m_k}$ with all $n_i, m_i \in \Z$ where all but possibly $n_1$ and / or $m_k$ are non-zero. Then set 
\[
\Phi(w) = \at^{\sign(n_1)} \bt^{\sign(m_1)} \cdots \at^{\sign(n_k)} \bt^{\sign(m_k)}
\]
where $\sign \col \Z \to \{ +1, 0, -1 \}$ is defined as usual. This may be seen to be a letter-quasimorphism and will be vastly generalised to amalgamated free products; see Lemma \ref{lemma:amalgamated yields letter-quasimorphism}.
Observe that for any group $G$ and any homomorphism $\Omega \col G \to \F_2$ the map $\Phi \circ \Omega \col G \to \Acl$ is a letter-quasimorphism.
Suppose that $G$ is \emph{residually free}. Then for ever non-trivial element $g \in G$ there is a homomorphism $\Omega_g \col G \to \F_2$ such that $\Omega_g(g) \in \F_2$ is nontrivial.
By applying a suitable automorphism on $\F_2$ to $\Omega_g$ we may assume that $\Omega_g(g)$ starts in a power of $\at$ and ends in a power of $\bt$. Then $\Phi_g := \Phi \circ \Omega_g$ is a letter quasimorphism such that $\Phi_g(g)$ is nontrivial and such that $\Phi_g(g^n) = \Phi_g(g)^n$.
\end{exmp}

\begin{defn} \label{defn:well behaved letter quasimorphisms}
We will call triples $(x_1,x_2,x_3)$ \emph{degenerate} if they are equivalent to a triple $(w, w^{-1}, e)$ for some $w \in \Acl$.

Let $G$ be a group. A map $\Psi \col G \to \Acl$ is called \emph{well-behaved letter-quasimorphism} if $\Psi$ is alternating, i.e. $\Psi(g^{-1}) = \Psi(g)^{-1}$ for every $g \in G$, and for all $g,h \in G$, the triple 
\[
(\Psi(g), \Psi(h), \Psi(gh)^{-1})
\]
is either 
letter-thin (see Definition \ref{defn:letter-thin}) or degenerate.
\end{defn}

\begin{rmk} \label{prop: alpha and beta preserve well-behaved letter-quasimorphisms}
Note that a triple $(x_1,x_2,x_3)$ is degenerate if and only if there is some $w \in \Acl$ such that
$(x_1,x_2,x_3)$ equals $(w,w^{-1},e)$, $(w,e, w^{-1})$ or $(e,w,w^{-1})$.
Note that if $\Phi \col G \to \Acl$ is a well-behaved letter-quasimorphism then also $\alpha \circ \Phi \col G \to \Acl$ and 
$\beta \circ \Phi \col G \to \Acl$ are well-behaved letter-quasimorphisms.
This follows immediately from Lemma \ref{lemma:alpha keeps thin.} and the fact that $\alpha$ (resp. $\beta$) satisfies $\alpha(w^{-1}) = \alpha(w)^{-1})$ (resp. $\beta(w^{-1}) = \beta(w)^{-1})$) for any $w \in \Acl$.
\end{rmk}

It is easy to see that every well-behaved letter-quasimorphism is also a letter-quasimorphism. 
The contrary does not hold.
The map $\Phi \col \F_2 \to \Acl$ described in Example \ref{exmp:letter quasimorphisms on free group}
is a letter-quasimorphism but not a well-behaved letter-quasimorphism. For example for $g= \at$, $h = \at$ we obtain
$(\Phi(g), \Phi(h), \Phi(h^{-1} g^{-1})) = (\at, \at, \at^{-1})$, which is neither letter-thin nor degenerate.

However, we may 
assign to each letter-quasimorphism $\Phi$ a well-behaved letter-quasimorphism $\tilde{\Phi}$.
This will be done 
by pre-composing $\Phi$ with a map $w \mapsto \tilde{w}$ defined as follows.

 Set $\tilde{w} = e$ whenever  $w \in \{ \at, e, \at^{-1} \}$. Else let $\zt_s$ be the first and $\zt_e$ be the last letter of $w \in \Acl$. Define $\tilde{w}$ as the reduced element in $\F_2$ freely equal to
$\tilde{w} := \zeta_s(\zt_s) w \zeta_e(\zt_e)$ where 
\[
\zeta_s(\zt) =
\begin{cases} e & \text{ if } \zt=\at \\
\at &  \text{ if } \zt = \bt \text{ or } \bt^{-1} \\
\at^2 & \text{ if } \zt = \at^{-1}
\end{cases}
\]
and
\[
\zeta_e(\zt) =
\begin{cases} e & \text{ if } \zt = \at^{-1} \\
\at^{-1} & \text{ if } \zt =  \bt \text{ or } \bt^{-1} \\
\at^{-2} & \text{ if } \zt = \at.
\end{cases}
\]
The key point is that $\tilde{w}$ starts with $\at$ and ends with $\at^{-1}$, unless $w \in \{ \at, e, \at^{-1} \}$. 
Observe that $\zeta_e(\zt)^{-1} = \zeta_s(\zt)$, and hence the map $w \mapsto \tilde{w}$ is alternating, i.e. $\widetilde{w^{-1}}= \tilde{w}^{-1}$.
For example, $\at \mapsto e$, $\at \bt \at^{-1} \mapsto \at \bt \at^{-1}$  and $\at^{-1} \bt \at \bt \at \mapsto \at \bt \at \bt \at^{-1}$.

If $\Phi \col G \to \Acl$ is a letter-quasimorphism then we define $\tilde{\Phi} \col G \to \Acl$ via $\tilde{\Phi}(g) := \widetilde{\Phi(g)}$.
\begin{prop} \label{prop:every letter-qm induces well behaved}
If $\Phi \col G \to \Acl$ is a letter-quasimorphism then $\tilde \Phi \col G \to \Acl$ is a well-behaved letter-quasimorphism, called the \emph{associated} well-behaved letter-quasimorphism.
\end{prop}

\begin{proof}
As $w \mapsto \tilde{w}$ commutes with taking inverses, if $\Phi$ is alternating then so is $\tilde{\Phi}$. 
In what follows we will use the following easy to check claim.
\begin{claim}
Let $(x_1,x_2,x_3)$ be an arbitrary triple obtained from $(y_1,y_2,y_3)$ by applying a sequence of the equivalences $(i)$ and $(ii)$ of Definition \ref{defn:equivalent triples}. Then $(\tilde{x}_1, \tilde{x}_2, \tilde{x}_3) \sim (\tilde{y}_1,\tilde{y}_2,\tilde{y}_3)$. In this case we say that the triples $(x_1,x_2,x_3)$ and $(y_1,y_2,y_3)$ are \emph{equivalent up to rotation and inverses}.
\end{claim}
Let $g,h \in G$. We wish to show that 
$(\tilde{\Phi}(g), \tilde{\Phi}(h), \tilde{\Phi}(gh)^{-1})$ is a letter-thin triple or degenerate, i.e. equivalent to $(w, w^{-1}, e)$ for some $w \in \Acl$.
If $(\Phi(g), \Phi(h), \Phi(gh)^{-1})$ is equivalent up to rotation and inverses to $(u_1,u_2,u_3)$ the above claim implies that it suffices to check that $(\tilde{u}_1,\tilde{u}_2,\tilde{u}_3)$
is either letter-thin or equivalent to $(w, w^{-1}, e)$.

First suppose that $g,h$ are as in Case (\ref{defn:letter-qm:thin}) of Definition \ref{defn:letter quasihomomorphism} i.e.
 $\Phi(g) \Phi(h) \Phi(gh)^{-1} = e$. 
If one of $\Phi(g)$, $\Phi(h)$ and $\Phi(gh)$ are trivial then the two other elements are inverses.
Hence, up to rotation and taking inverses we may assume that 
\[
(\Phi(g), \Phi(h), \Phi(gh)^{-1})=(u,u^{-1},e)
\]
for some $u \in \Acl$. 
Hence $(\tilde{u}, \tilde{u}^{-1}, e)$ is degenerate.

If none of $\Phi(g)$, $\Phi(h)$ and $\Phi(gh)^{-1}$ are trivial then, as $\Phi$ maps to alternating elements, there are elements $u_1, u_2$ such that $u_1$ ends in a power of $\at$ and $u_2$ starts in a power of $\bt$, such that 
$(\Phi(g), \Phi(h), \Phi(gh))$ is equivalent up to rotation and taking inverses to $(u_1, u_2, u_3)$ where 
$u_3 = u_2^{-1} u_1^{-1}$ as a reduced word. 
Further, write $u_1 = u_1' \xt$ as a reduced word for $\xt \in \{ \at, \at^{-1} \}$ and an appropriate word $u_1' \in \Acl$. If $u_1'$ is empty, then $\tilde{u}_1=e$. Let $\zt_2$ be the last letter of $u_2$. Then
\[
(\tilde{u}_1, \tilde{u}_2, \tilde{u}_3) = (e, \at u_2 \zeta_e(\zt_2), \zeta_e(\zt_2)^{-1} u_2^{-1} \at^{-1} )
\]
which is equivalent to $(w, w^{-1}, e)$ for $w = \at u_2 \zeta_e(\zt_2)$.
If $u_1'$ is non-empty, let $\zt_1$ be the first letter of $u_1'$ and as before let $\zt_2$ be the last letter of $u_2$.
Then
\[
(\tilde{u}_1, \tilde{u}_2, \tilde{u}_3) = (\zeta_s(\zt_1) u_1' \at^{-1} , \at u_2 \zeta_e(\zt_2), \zeta_e(\zt_2)^{-1} u_2^{-1}  \xt^{-1} u'^{-1}_1 \zeta_s(\zt_1)^{-1} )
\]
which can be seen to be letter-thin of type $\Toa$.
This shows that $(\tilde{\Phi}(g), \tilde{\Phi}(h), \tilde{\Phi}(gh)^{-1})$ is letter-thin or degenerate if $\Phi(g) \Phi(h) \Phi(gh)^{-1} = e$.

Hence, suppose that $g,h$ are as in Case (\ref{defn:letter-qm:general}) of Definition \ref{defn:letter quasihomomorphism}.
Then $(\Phi(g), \Phi(h), \Phi(gh))$  is equivalent up to rotation and inverses to 
\[
(u_1,u_2,u_3) = (c_1^{-1} \xt c_2, c_2^{-1} \xt c_3, c_3^{-1} \xt^{-1} c_1) 
\]
for $\xt \in \{ \at, \bt \}$ where $c_1,c_2,c_3 \in \Acl$ are \emph{arbitrary} i.e. we do not assume that $c_2$ is non-empty as in Definition \ref{defn:letter-thin}.
First, suppose that $\xt = \bt$.
Define
\[
d_i = \begin{cases} 
c_i \zeta_e(\zt_i) & \text{ if } c_i \not = e  \\
 \at^{-1} & \text{ else} 
 \end{cases}
\]
where $\zt_i$ is the last letter of $c_i$. We may see then, that 
\[
(\tilde{u}_1, \tilde{u}_2, \tilde{u}_3) = (d_1^{-1} \bt d_2, d_2^{-1} \bt d_3, d_3^{-1} \bt^{-1} d_1) 
\]
which is letter thin of type $\Tob$ as all $d_i$'s are non trivial.

Hence, suppose that $\xt = \at$. For what follows, if $c_i$ is non-empty, we will denote by $\zt_i$ the last letter of $c_i$ and let $d_i$ be the freely reduced word represented by $c_i \zeta_e(\zt_i)$. Observe that if $c_i$ is non-empty then so is $d_i$.

There are the following cases:
\begin{itemize}
\item[(i)] $c_1 \not = e$, $c_2 \not = e$, $c_3 \not = e$: Then
$(\tilde{u}_1, \tilde{u}_2, \tilde{u}_3) = (d_1^{-1} \at d_2, d_2^{-1} \at d_3, d_3^{-1} \at^{-1} d_1) $
\item[(ii)] $c_1 \not = e$, $c_2 \not = e$, $c_3 = e$: Then
$(\tilde{u}_1, \tilde{u}_2, \tilde{u}_3) = (d_1^{-1} \at d_2, d_2^{-1} \at^{-1} ,  \at d_1) $
\item[(iii)] $c_1 \not = e$, $c_2  = e$, $c_3 \not = e$: Then
$(\tilde{u}_1, \tilde{u}_2, \tilde{u}_3) = (d_1^{-1} \at^{-1} ,  \at d_3, d_3^{-1} \at^{-1} d_1) $
\item[(iv)] $c_1  = e$, $c_2 \not = e$, $c_3 \not = e$: Then
$(\tilde{u}_1, \tilde{u}_2, \tilde{u}_3) = (\at d_2, d_2^{-1} \at d_3, d_3^{-1} \at^{-1} ) $
\item[(v)] $c_1 \not = e$, $c_2  = e$, $c_3  = e$: Then
$(\tilde{u}_1, \tilde{u}_2, \tilde{u}_3) = (d_1^{-1} \at, e ,  \at^{-1} d_1) $
\item[(vi)] $c_1  = e$, $c_2 \not = e$, $c_3  = e$: Then
$(\tilde{u}_1, \tilde{u}_2, \tilde{u}_3) = ( \at d_2, d_2^{-1} \at^{-1} , e ) $
\item[(vii)] $c_1  = e$, $c_2  = e$, $c_3 \not = e$: Then
$(\tilde{u}_1, \tilde{u}_2, \tilde{u}_3) = ( e, \at d_3, d_3^{-1} \at^{-1} ) $
\item[(viii)] $c_1  = e$, $c_2  = e$, $c_3  = e$: Then
$(\tilde{u}_1, \tilde{u}_2, \tilde{u}_3) = (e, e, e) $
\end{itemize}
and 
cases $(i)-(iv)$ can be seen to be letter-thin of type $\Toa$ and cases $(v)-(viii)$ can be seen to be degenerate. This completes the proof.
\end{proof}

Both letter-quasimorphisms and well-behaved letter-quasimorphisms are examples of \emph{quasimorphism} in the sense of Hartnick--Schweitzer \cite{hartnick-schweitzer}; see Subsection \ref{subsec:generalised qm}.
Let $\Phi$ be a letter-quasimorphism and let $\bar{\eta} \col \F_2 \to \R$ be an ordinary homogeneous quasimorphism with defect $D$ which vanishes on the generators $\at, \bt$.
We wish to calculate the defect of $\bar{\eta}\circ \Phi$. Fix $g,h \in G$. If $\Phi(g) \Phi(h) = \Phi(gh)$, then
\[
|\bar{\eta} \circ \Phi(g) + \bar{\eta} \circ \Phi(h) - \bar{\eta} \circ \Phi(gh) | \leq D
\]
Else, up to rotating the factors we see that  
\[
(\Phi(g), \Phi(h), \Phi(gh)^{-1}) = (d_1^{-1} \xt d_2, d_2^{-1} d_3, d_3^{-1} d_1)
\] for some appropriate $d_1,d_2,d_3 \in \Acl$, $\xt \in \{ \at, \at^{-1}, \bt, \bt^{-1} \}$. Then, as $\bar{\eta}$ is homogeneous  $\bar{\eta}(d_1^{-1} \xt d_2) =
\bar{\eta}(\xt d_2 d_1^{-1})$ and hence $|\bar{\eta}(\xt d_2 d_1^{-1})- \bar{\eta}(d_2 d_1^{-1})| \leq D$ as we assumed that $\bar{\eta}$ vanishes on the generators. Then we may estimate
\[
|\bar{\eta} \circ \Phi(g) + \bar{\eta} \circ \Phi(h) + \bar{\eta} \circ \Phi(gh)^{-1} | = |\bar{\eta}(d_1^{-1} \xt d_2) + \bar{\eta}(d_2^{-1} d_3) + \bar{\eta}(d_3^{-1} d_1)| \leq 4D
\]
and after homogenisation of $\phi = \bar{\eta} \circ \Phi(g)$ we estimate that $D(\bar{\phi}) \leq 8D$ using that homogenisation at most doubles the defect; see Proposition \ref{prop:defect of homogenisation doubles}.
Hence if $\Phi(g) \in \F_2'$ is such that $\Phi(g^n) = w^n$ for some non-trivial $w \in \Acl$ which also lies in the commutator subgroup $\F'$ and $\eta \col \F_2 \to \R$ is homogenous and extremal to $\Phi(g)$ with defect $1$ then, by Bavard, \[
\scl(g) \geq \frac{\bar{\phi}(g)}{16} \geq \frac{\bar{\eta}(\Phi(g))}{16} = \frac{\scl(\Phi(g))}{8}
\]
and in particular $\scl(g) \geq 1/16$.
This is already a good estimate but we see that we can do much better; see Theorem \ref{thm:main}. 

We will see that this notion is much more flexible than homomorphisms. There are groups $G$ such that for every non-trivial element $g \in G'$ there is a letter-quasimorphisms $\Phi$ such that $\Phi(g)$ is non-trivial.
This may be possible even if the group $G$ is not residually free, for example if $G$ is a right-angled Artin group; see Section \ref{sec:RAAGs and scl}.

\subsection{Main Theorem}

We now deduce our main criterion for $1/2$-gaps in $\scl$:
\begin{thm} \label{thm:main}
Let $G$ be a group and let $g_0 \in G$. Suppose there is a letter-quasimorphism $\Phi \col G \to \Acl$ such that $\Phi(g_0)$ is non-trivial and that  $\Phi(g_0^n) = \Phi(g_0)^n$ for all $n \in \N$.
Then there is an explicit homogeneous quasimorphism $\bar{\phi} \col G \to \R$ with $D(\bar{\phi}) \leq 1$ such that $\bar{\phi}(g_0) \geq 1$.
If $g_0 \in G'$, then $\textrm{scl}(g_0) \geq 1/2$. 

If $G$ is countable then there is an action $\rho \col G \to \mathrm{Homeo}^+(S^1)$ such that $[\delta^1 \bar{\phi}]=\rho^*\eurm^\R_b \in \Hrm^2_b(G,\R)$, for $\eurm^\R_b$ the real bounded Euler class.
\end{thm}
In particular, the $\Phi(g_0) \in \Acl$ of the Theorem has to be alternating and of \emph{even length}, else $\Phi(g_0)^n$ would not be an alternating word.

\begin{proof}
Let $\Phi \col G \to \Acl$ be the letter-quasimorphism as in the theorem and let $\tilde{\Phi} \col G \to \Acl$ be the associated well behaved letter-quasimorphism described above.
As $\tilde \Phi(g_0)$ is obtained from $\Phi(g_0)$ by just possibly changing the beginning and the end of the word $\Phi(g_0)$, it is easy to see that there 
are words $c_1, c_2, w \in \mathcal{A}$ such that $\tilde \Phi(g_0^n) = c_1^{-1} w^{n-1} c_2$ as a freely reduced word for all $n \geq 1$.

Consider the sequence $\gamma_i$ of maps $\gamma_i \col \mathcal{A} \to \mathcal{A}$ 
defined via $\gamma_0 = id$, $\gamma_{2k+1} = (\alpha \circ \beta )^k \circ \alpha$ and $\gamma_{2k} = (\beta \circ \alpha)^k$ and note that $\gamma_i$ is either $\alpha \circ \gamma_{i-1}$ or $\beta \circ \gamma_{i-1}$; see Definition \ref{defn:alpha and beta}.
Analogously define the sequence $\bar{\gamma}_i \col \bar{\Acl}_0 \to \bar{\Acl}_0$ of maps via
$\bar{\gamma}_0 = id$, $\bar{\gamma}_{2k+1} = (\bar{\alpha} \circ \bar{\beta} )^k \circ \bar{\alpha}$ and $\bar{\gamma}_{2k} = (\bar{\beta} \circ \bar{\alpha})^k$ and note that every $\bar{\gamma}_i$ is either $\bar{\alpha} \circ \bar{\gamma}_{i-1}$ or $\bar{\beta} \circ \bar{\gamma}_{i-1}$; see Definition \ref{defn:maps alpha bar and beta bar}.
For every letter-thin triple $(x_1, x_2, x_3)$ also $\gamma_i(x_1, x_2, x_3)$ is letter-thin by multiple applications of Lemma \ref{lemma:alpha keeps thin.}. Furthermore, if $(x_1, x_2, x_3)$ is a degenerate triple as in Definition \ref{defn:well behaved letter quasimorphisms}, then also $\gamma_i(x_1,x_2,x_3)$ is a degenerate triple as $\gamma_i$ satisfies $\gamma_i(x^{-1}) = \gamma_i(x)^{-1}$ for all $x \in \Acl$.

Let $w$ be as above and consider the sequence $\bar{\gamma}_i(w) \in \bar{\Acl}_0$ of conjugacy classes in $\bar{\Acl}_0$. By Proposition \ref{prop:alpha on conjugacy classes decreases},
 if $\bar{\gamma}_i(w)$ is a non-trivial equivalence class in the commutator subgroup then $\bar{\gamma}_{i+1}(w)$ either is non-trivial and has strictly smaller word-length or $\bar{\gamma}_{i}(w) = \bar{\gamma}_{i+1}(w)$; see also Remark \ref{rmk:on conjugacy classes for acl}.

Hence, there are the following cases:
\begin{itemize}
\item For all $i \in \N$, $\bar{\gamma}_i(w)$ lies in $\F_2'$, the commutator subgroup. Then, there is an $N$ such that $\bar{\gamma}_{N}(w) = \bar{\gamma}_{N+i}(w)$ for all $i \in \N$. Both $\bar{\alpha}$ and $\bar{\beta}$ then fix the class $\bar{\gamma}_N(w)$.
By Proposition \ref{prop:alpha on conjugacy classes decreases}, $\bar{\gamma}_N(w)$ may be represented by
$[\at, \bt]^k$ for $k \in \Z \backslash \{ 0 \}$.
Hence, the quasimorphism $\eta_0 = \eta_{\at \bt} - \eta_{\bt \at}$ studied in Example \ref{exmp: extemal brooks quasimorphisms on free group} and Proposition \ref{prop:letter thin triples and two quasimorphisms},
satisfies that $|\bar{\eta}_0(\bar{\gamma}_N(w))| \geq 2$. 
Define $\psi \col G \to \Z$ via
\[
\psi(g) := \begin{cases} 
\eta_0 \circ \gamma_N \circ \tilde \Phi(g) & \text{ if } \gamma_N \circ \tilde \Phi(g) \not = e \\
1 & \text{ else}
\end{cases}
\]
and observe that if $\gamma_N \circ \tilde \Phi(g)$ is non-trivial, then $\psi(g^{-1}) = - \psi(g)$.
By multiple applications of Proposition \ref{prop:powers of alpha, beta},  we see that there are some elements 
$d_1, d_2, w' \in \Acl$ such that $\gamma_N \circ \tilde \Phi(g^n) = d_1 w'^{n-K} d_2$ for all $n \geq K$, for $K \leq N+1$ and $[w'] = \bar{\gamma}_N([w])$.
We see that 
\begin{align*}
|\bar{\psi}(g_0)| &= \lim_{n \to \infty} |\psi(g_0^n)|/n \\
&= \lim_{n \to \infty} |\eta_0 \circ \gamma_N \circ \tilde \Phi(g_0^n)| / n \\
&= \lim_{n \to \infty} |\eta_0 (d_1 w'^{n-K} d_2)|/n \\
&= |\bar{\eta_0} (\bar{\gamma}_N([w]))| \geq 2.
\end{align*}

By multiple applications of Lemma \ref{lemma:alpha keeps thin.} and the fact that $\alpha(w^{-1}) = \alpha(w)^{-1}$, $\beta(w^{-1}) = \beta(w)^{-1}$ and $\alpha(e)=e=\beta(e)$  we see that $\gamma_N \circ \tilde \Phi$ is a well-behaved letter-quasimorphism.
Let $g,h \in G$. We wish to compute the defect $| \psi(g) + \psi(h) - \psi(gh) |$. To ease notation define $(x_1,x_2,x_3)$ as the triple
\begin{align*}
(x_1,x_2,x_3) = (\gamma_N \circ \tilde \Phi(g), \gamma_N \circ \tilde \Phi(h), \gamma_N \circ \tilde \Phi(gh)^{-1})
\end{align*}
which is either letter-thin or degenerate as $\gamma_N \circ \tilde \Phi$ is a well-behaved letter-quasimorphism.
If $(x_1,x_2,x_3)$ letter-thin then none of its components $x_i$ are empty. Hence
\begin{equ*}{rcl}
|\psi(g)+\psi(h)-\psi(gh)| &=& |\psi(g)+\psi(h)+\psi(h^{-1} g^{-1})| \\
&=& |\eta_0(x_1)+ \eta_0(x_2) + \eta_0(x_3)| \\
&=& 1
\end{equ*}
by Proposition \ref{prop:letter thin triples and two quasimorphisms}.
Suppose that $(x_1,x_2,x_3)$ is degenerate.
Then one may see that $(x_1,x_2,x_3)$ equals $(v, v^{-1}, e)$, $(v,e, v^{-1})$ or $(e, v, v^{-1})$ for some $v \in \Acl$. Using that $-\eta_0(v) = \eta_0(v^{-1})$ for $e \not = v \in \Acl$ we see that two terms of $\psi(g) + \psi(h) - \psi(gh)$ will cancel and for the other will be $1$. Hence, $|\psi(g) + \psi(h) - \psi(gh)| =1$.
Finally, if $(x_1,x_2,x_3) = (e,e,e)$ then $\psi(g) + \psi(h) - \psi(gh) = 1$.
In particular we see that for any $g,h \in G$, $\psi(g) + \psi(h) - \psi(gh) \in \{ 1, -1 \}$, so $\psi$ is a quasimorphism.
Moreover, by possibly changing the sign of $\psi$ we may assume that $\bar{\psi}(g_0) \geq 2$.

\item Otherwise, let $N \in \N$ be the smallest integer such that $\bar{\gamma}_N(w) \not \in \F_2'$.  Then $\bar{\gamma}_N(w) \in \Acl$ is represented by a non-trivial even word which is not in the commutator. Hence
\[
|\eta_\at(\bar{\gamma}_N(w))| + |\eta_\bt(\bar{\gamma}_N(w))| \geq 2
\]
where $\eta_\at \col \F_2 \to \Z$ (resp. $\eta_\bt \col \F_2 \to \Z$) denotes the homomorphism counting the letter $\at$ (resp. $\bt$).
 Observe that homomorphisms are already homogenised. There is some $\eta = \eta_\xt + \eta_\yt$ where $\xt \in \{ \at, \at^{-1} \}$, $\yt \in \{ \bt, \bt^{-1} \}$ such that
$\eta(\bar{\gamma}_N(w)) \geq 2$. As before, define $\psi \col G \to \Z$ via
\[
\psi(g) := \begin{cases} 
\eta\circ \gamma_N \circ \tilde \Phi(g) & \text{ if } \gamma_N \circ \tilde \Phi(g) \not = e \\
1 & \text{ else}.
\end{cases}
\]
By a similar argument as above we see that $\bar{\psi}(g_0) \geq 2$.
Again, the triple 
\begin{align*}
(x_1,x_2,x_3) = (\gamma_N \circ \tilde \Phi(g), \gamma_N \circ \tilde \Phi(h), \gamma_N \circ \tilde \Phi(h^{-1} g^{-1}))
\end{align*}
is either letter-thin or degenerate.
By the same argument as in the previous case and using Proposition \ref{prop: letter thin triples and homomorphisms} we conclude that for any $g,h \in G$, $|\psi(g) + \psi(h) - \psi(gh)|=1$, so $\psi$ is a quasimorphism.
In particular we see that for any $g,h \in G$, $\psi(g) + \psi(h) - \psi(gh) \in \{ 1, -1 \}$.
\end{itemize}

In both cases, set 
\[
\phi(g) :=  \frac{\psi(g)+1}{2}.
\]
 Then we see that, for any $g,h \in G$,
\[
\delta^1 \phi(g,h) =  \phi(g) + \phi(h) - \phi(gh)  = \frac{\psi(g) + \psi(h) - \psi(gh)+1}{2}
\in \{ 0, 1 \}.
\]
Hence, by Theorem \ref{thm:ghys} due to Ghys (see also \cite{ghys}), there is an action $\rho \col G \to \mathrm{Homeo}^+(S^1)$ on the circle such that $\rho^*\eurm_b = [\delta^1 \phi] \in \mathrm{H}^2_b(G,\Z)$ and hence $\rho^*\eurm^\R_b = [\delta^1 \bar{\phi}] \in \Hrm^2_b(G, \R)$. Here, $\eurm_b$ (resp. $\eurm_b^\R$) denotes the (real) bounded Euler class.
Moreover, we observe that $\bar{\phi}(g) = \bar{\psi}(g)/2$, for $\bar{\phi}$ the homogenisation of $\phi$.
Furthermore, as $D(\psi) = 1$ we estimate by Proposition \ref{prop:defect of homogenisation doubles}
that $D(\bar{\psi}) \leq 2$ and hence $D(\bar{\phi}) \leq 1$.

We conclude that there is a quasimorphism $\phi \col G \to \R$ with homogenisation $\bar{\phi}$ such that
$D(\bar{\phi}) \leq 1$, $\bar{\phi}(g_0) \geq 1$. If $G$ is countable then there is an action $\rho \col G \to \mathrm{Homeo}^+(S^1)$ with $[\delta^1 \phi] = \rho^*\eurm_b^\R \in \mathrm{H}^2_b(G,\R)$ where $\eurm_b^\R$ is the real bounded Euler class.
\end{proof}

Applying Theorem \ref{thm:main} to Example \ref{exmp:letter quasimorphisms on free group} we recover that in every residually free group $G$, every non-trivial element $g \in G'$ has stable commutator length  at least $1/2$. This gap is realised by a quasimorphism induced by a circle action which has not been known previously.

As said in the introduction we think of letter-quasimorphisms as simplifications of elements. 
Sometimes information about $w$ can not be recovered by $\Phi(w)$. 
For example for the word $w = \at \bt \at^{-1} \bt^{-1} \at \bt^{-3} \at^{-1} \bt^3$, we may compute\footnote{These calculations are done with \texttt{scallop}, see \cite{scallop}}
$\scl(w) = 3/4$ but $\scl(\Phi(w)) = 1/2$.
This example may be generalised: Pick an alternating word $w \in \Acl$ that starts and ends in a power of $\bt$.
Then $[\at, w] \in \Acl$ and $\scl([\at, w]) = 1/2$.
Then for any choice of words $v_1, v_2 \in \F_2$ such that $\Phi(v_1) = w$, $\Phi(v_2) = w^{-1}$ and such that $v = \at v_1 \at^{-1} v_2 \in \F_2'$ we have that $\Phi(v) = [\at, w]$. However, $\scl(v)$ is experimentally arbitrarily large.

\begin{rmk} \label{rmk:quasimorphisms are pullback of hs qm}
As pointed out in the proof all of $\gamma_i \circ \tilde \Phi$ are well-behaved letter-quasimorphisms for any $i \in \N$.
The quasimorphisms $\psi$ defined in the proof are then pullbacks of the quasimorphism $\eta_0 = \eta_{\at \bt} - \eta_{\bt \at}$ or homomorphisms $\eta = \eta_\xt + \eta_\yt$ via these well-behaved letter-quasimorphisms $\gamma_i \circ \tilde{\Phi} \col G \to \Acl \subset \F_2$.
\end{rmk}

\begin{rmk} \label{rmk:criterion for gaps}
In light of Theorem \ref{thm:Bavards duality}, a criterion for groups to have the optimal $\scl$-gap of $1/2$ may hence be as follows:

\begin{center}
\emph{Let $G$ be a non-abelian group. If for every non-trivial element $g \in G'$ there is a letter-quasimorphism $\Phi \col G \to \Acl$ such that $\Phi(g^n) = \Phi(g)^n$ where $\Phi(g)$ is non-trivial. Then $G$ has a gap of $1/2$ in stable commutator length.}
\end{center}

By Example \ref{exmp:letter quasimorphisms on free group} residually free groups have this property and the criterion has some 
qualitative similarities to being residually free.
We will later see that also non-residually free groups, like right-angled Artin groups, have this property; see Section \ref{sec:RAAGs and scl}.
\end{rmk}

\section{Left Orders and Left-Relatively Convex Subgroups} \label{sec:Left orders and convex subgroups}

For what follows we will use the notation and conventions of \cite{convexsub}. We further emphasise that nothing in this section is original work.

An order $\prec$ on a set $\mathcal{X}$ is a subset of  $\mathcal{X} \times \mathcal{X}$ where we 
stress that a pair $(x,y) \in \mathcal{X} \times \mathcal{X}$ is in this subset by writing $x \prec y$. Furthermore, the following holds:
\begin{itemize}
\item For all $x, y \in \mathcal{X}$ either $x \prec y$ or $y \prec x$. We have $x \prec y$ and $y \prec x$ if and only if $x = y$.
\item For all $x, y, z \in \mathcal{X}$ such that $x \prec y$ and $y \prec z$ we have $x \prec z$.
\end{itemize}

A set $\mathcal{X}$ with a left group action has a \emph{$G$-invariant order} if for all $g \in G$, $x_1,x_2 \in \mathcal{X}$, 
$x_1 \prec x_2$ implies that $g.x_1 \prec g.x_2$. 
A group $G$ is said to be \emph{left orderable} if the set $G$ has a $G$-invariant order with respect to its left action on itself. A subgroup $H < G$ is said to be \emph{left relatively convex} in $G$ if the $G$-set $G/H$ has some $G$-invariant order. Note that this definition is valid even if $G$ itself is \emph{not} left-orderable.
If $G$ itself is orderable, then this is equivalent to the following: There is an order $\prec$ on $G$ such that for every $h_1, h_2 \in H$ and $g \in G$ with $h_1 \prec g \prec h_2$ we may conclude $g \in H$. In this case we simply say that $H$ is convex in $G$. As $e \in H$, this means that $H$ is a neighbourhood of $e$.
It is not hard to see that left relatively convex is transitive:
\begin{prop} \label{prop:left-relatively convex is transitive} \footnote{See Section 2 of \cite{convexsub}}
Let $K < H < G$ be groups. Then $G/K$ is $G$-orderable such that $H/K$ is convex if and only if $G/H$ is $G$-orderable and $H/K$ is $H$-orderable. 
\end{prop}
An easy example of a pair $H < G$  such that $H$ is left relatively convex in $G$ is $\Z < \Z^2$ embedded in the second coordinate via the standard lexicographic order. Similarly, every subgroup $G < \Z \times G$ embedded via the second coordinate, is left relatively convex for an arbitrary group $G$.
Every generator of a non-abelian free group generates a left relatively convex subgroup in the total group; see \cite{DH}. In fact, \cite{convexsub} show that each maximal cyclic subgroup of a right-angled Artin group is left relatively convex.

We wish to state the main Theorem of \cite{convexsub}. For this let $\mathrm{T}$ denote an \emph{oriented} simplicial tree, with vertices $\Vrm(\mathrm{T})$ and edges $\Erm(\mathrm{T})$ and two maps $\iota, \tau \col \Erm(\mathrm{T}) \to \Vrm(\mathrm{T})$ assigning to each oriented edge its initial and terminal vertex respectively.
Suppose that $G$ acts on $\mathrm{T}$ and denote by $G_v$ (resp. $G_e$) the stabilisers of a vertex $v \in \mathrm{V}(\Trm)$ (resp. an edge $e \in \Erm(\Trm)$).
Note that stabilisers of an edge $e$ naturally embed into $G_{\iota(e)}$ and $G_{\tau(e)}$.
\begin{thm} \footnote{Theorem 14 of \cite{convexsub}}
Suppose that $\mathrm{T}$ is a left G-tree such that, for each $\mathrm{T}$-edge $e$, $G_e$ is left relatively convex in $G_{\iota(e)}$ and in $G_{\tau(e)}$. Then, for each $v \in \mathrm{V}(\mathrm{T})$, $G_v$ is left relatively convex in $G$.
Moreover, if there exists some $v \in \mathrm{V}(\mathrm{T})$ such that $G_v$ is left orderable, then $G$ is left orderable.
\end{thm}
We deduce the following corollary, see Example 19 of \cite{convexsub} using Bass--Serre Theory.
\begin{corr} \label{corr:orders of amalgamations}
Let $A, B$ and $C$ be groups and let $\kappa_A \col C \inj A$ and $\kappa_B \col C \inj B$ be injections and let 
$G = A \star_C B$ be the corresponding amalgamated free product (see Section \ref{sec:amalgamation}).
If $\kappa_A(C)$ is left relatively convex in $A$ and $\kappa_B(C)$ is left relatively convex in $B$, then $A$ and $B$ are left relatively convex in $G$. 
\end{corr}

Let $H<G$ be a left relatively convex subgroup and let $\prec$  be a $G$-invariant order of $G/H$. we define the \emph{sign-function} $\sign \col G \to \{ -1, 0, 1 \}$ on representatives $g \in G$ of cosets in $G/H$ via
\[
\textrm{sign}(g) =
\begin{cases} +1 & \text{ if } gH \succ H \\
0 & \text{ if } g \in H \\
-1 & \text{ if } gH \prec H
\end{cases}
\]
\begin{prop} \label{prop:orders well defined}
Let $H < G$ be a left relatively convex subgroup and let $\prec$ be the $G$-invariant order of $G/H$.
Then the sign-function with respect to $\prec$ on elements in $G$ is independent under left or right multiplication by elements of $H$. That is for every $g \in G \smallsetminus H$ and
for every $h \in H$, $\textrm{sign}(h g) = \textrm{sign}(g) = \textrm{sign}(g h)$.
\end{prop}

\begin{proof}
Clearly $\sign(gh) = \sign(g)$ as both $g$ and $gh$ define the same coset. 
On the other hand, if $h g H \succ H$ then by left multiplication $g H \succ H$ and similarly
if $h g H \prec H$ then $g H \prec H$, so $\sign(hg) = \sign(g)$.
\end{proof}

\section{Amalgamted Free Products} \label{sec:amalgamation}

Let $A, B, C$ be groups and let $\kappa_A \col C \inj A$, $\kappa_B \col C \inj B$ be injections. The
\emph{amalgamated free product} $G =A \star_C B$ with respect to $\kappa_A$ and $\kappa_B$ it the group
via
\[
G = A \star_C B = A \star B / \langle \langle \kappa_A(c)^{-1} \kappa_B(c) \mid c \in C \rangle \rangle.
\]
It is a well-known fact that the homomorphism $A \to A \star_C B$ (resp. $B \to A \star_C B$) defined by mapping $a \in A$ (resp. $b \in B$) to the corresponding element $a \in G$ (resp. $b \in G$) is \emph{injective} and that $C$ embeds in $G$ via these injections.
See \cite{serre} for a reference.
Every element $g \in G$ with $g \in G \smallsetminus  C $ may be written as a product
\[
g = d_1 \cdots d_k
\]
such that all of $d_i$ are either in $A \smallsetminus \kappa_A(C)$ or in $B \smallsetminus \kappa_B(C)$ 
and alternate between both. Furthermore for any other such expression
\[
g = d'_1 \cdots d'_{k'}
\]
one may deduce that $k'=k$ and that there are elements $c_i \in C$, $i \in \{ 1, \ldots, k-1 \}$ such that
$d'_1 = d_1 c_1$, $d'_i = c_{i-1}^{-1} d_i c_i$ and $d'_k = c_{k-1} d_k$.

For what follows, let $\prec_A$ (resp. $\prec_B)$ be a left order on $A/\kappa_A(C)$ (resp. $B/\kappa_B(C)$) and let $\sign_A$ (resp. $\sign_B$) be its sign on $A$ (resp. $B$). 
We define the map $\Phi \col G \to \Acl$ as follows:
If $g \in C$ set $\Phi(g) = e$.
Else let $g = d_1 \cdots d_k$ be the normal form described above.
Then, set 
\[
\Phi(g) = \prod_{i=1}^k \Phi(d_i)
\]
where we define
 \[
 \Phi(d_i) = \begin{cases}
 \at^{ \textrm{sign}_A(d_i)} & \text{ if } d_i \in A \smallsetminus \kappa_A(C) \\
 \bt^{ \textrm{sign}_B(d_i)} & \text{ if } d_i \in B \smallsetminus \kappa_B(C)
 \end{cases}
 \] 
and we note that $\Phi$ is well defined.
To see this let $d'_1 \cdots d'_k$ be another normal form for $g$ and let $c_i \in C$ for $i \in \{0, \ldots, k+1 \}$ be such that 
$d'_i = c_{i-1}^{-1} d_i c_i$ with $c_0=c_{k+1}=e$.
Then 
\[
\sign(d_i) = \sign(c_{i-1}^{-1} d_i) = \sign(c_{i-1}^{-1} d_i c_i) = \sign(d'_i)
\]
by Proposition \ref{prop:orders well defined} and ``$\sign$'' either ``$\sign_A$'' or ``$\sign_B$''.

We claim that:
\begin{lemma} \label{lemma:amalgamated yields letter-quasimorphism}
Let $G = A \star_C B$ and $\Phi \col G \to \Acl$ be as above. Then $\Phi$ is a letter-quasimorphism.
\end{lemma}

We will prove this by giving another description of $\Phi$ in terms of paths in the Bass--Serre tree associated to the amalgamated free product $G = A \star_C B$:

Let $\mathrm{T}$ be the tree with vertex set 
$\Vrm (\Trm) = \{ g A \mid g \in G \} \sqcup \{ g B \mid g \in G \}$ and oriented edges
\[
\Erm(\Trm) = \{ (g A, g B) \mid g \in G \}  \sqcup \{ (g B, g A) \mid g \in G \} \subset \Vrm(\Trm) \times \Vrm(\Trm)
\]
We define $\iota, \tau \col \Erm(\Trm) \to \Vrm(\Trm)$ via $\iota((g A, g B)) = g A$, $\tau((g A, g B))= g B$ and similarly, $\iota(g B, g A) = g B$, $\tau(g B, g A)= g A$. Moreover, we set $(g A, g B)^{-1} = (g B, g A)$
and $(g B, g A)^{-1} = (g A, g B)$.
It is well-known that $\Trm$ is indeed a connected tree.

$G$ acts on $\Trm$ by left multiplication. We have that $\mathrm{Stab}_G(g A) = g A g^{-1} < G$, respectively 
$\mathrm{Stab}_G(h B) = h B h^{-1} < G$, $\mathrm{Stab}_G(g A, g B) = g C g^{-1}$ and $\mathrm{Stab}_G(g B, g A) = g C g^{-1}$

A \emph{reduced path of edges} is a sequence $\wp = (e_1, \ldots e_n)$, $e_i \in \Erm(\Trm)$ such that $\tau(e_i)=\iota(e_{i+1})$ for every $i \in \{ 1, \ldots, n-1 \}$, without backtracking. We call $n$ the \emph{length of the path}.
For what follows, $\mathcal{P}$ will be the set of all paths of edges.

We define the following map $\Xi \col \mathcal{P} \to \Acl$ assigning an alternating word to each path of edges.
Let $\wp \in \mathcal{P}$.
If $\wp$ has length $1$, then set $\Xi(\wp) :=e$.
Else, suppose that $\wp$ has length $2$, i.e. $\wp = (e_1,e_2)$.
Suppose that $e_1 = (g_1 A, g_1 B)$ and $e_2 = (g_2 B, g_2 A)$ and note that $g_1 B = g_2 B$. In particular, $g_1^{-1} g_2 \in B$.
Set $\Xi(\wp)=\Xi((e_1,e_2)) = \bt^{\sign_B(g_1^{-1} g_2)}$.
Similarly, if $e_1 = (g_1 B, g_1 A)$ and $e_2 = (g_2 A, g_2 B)$ note that $g_1 A = g_2 A$ and set $\Xi(\wp) = \Xi((e_1,e_2)) = \at^{\sign_A(g_1^{-1} g_2)}$.
Finally, for an arbitrary paths $\wp = (e_1, \ldots, e_n)$ set $\Xi(\wp) = \Xi(e_1,e_2) \cdot \Xi(e_2, e_3) \cdots \Xi(e_{n-2}, e_{n-1})\cdot \Xi(e_{n-1}, e_n)$.
Note that $\Xi$ is well defined. To see this, note that the stabilizer of any edge $(g A, g B)$ (resp. $(g B, g A)$) is $g C g^{-1}$. Hence, if $(g A, g B) = (g' A, g' B)$ (resp. $(g B, g A) = (g' B, g' A)$) there is a $c \in C$ such that $g c = g'$. If $(e_1,e_2)$ is a path of edges such that without loss of generality $e_1 = (g_1 A, g_1 B)= (g'_1 A, g'_1 B)$ and $e_2 = (g_2 A, g_2 B)=(g'_2 A, g'_2 B)$ then there are $c_1,c_2$ such that $g_1 = g'_1 c_1$ and $g_2 = g'_2 c_2$. Hence 
\[
\sign_B(g_1^{-1} g_2) = \sign_B(c_1^{-1} {g'_1}^{-1} g'_2 c_2) = \sign_B({g'_1}^{-1} g'_2)
\]
by Proposition \ref{prop:orders well defined}.
Define the \emph{inverse of a path} $\wp = (e_1, \ldots, e_n)$ as $\wp^{-1} := (e_n^{-1}, \ldots, e_1^{-1} )$. 
We see that $\Xi(\wp^{-1}) = \Xi(\wp)^{-1}$ using that $\sign(g^{-1}) = - \sign(g)$. 
We collect some further properties of $\Xi$.
We note that if $\wp \in \mathcal{P}$ is a path then so is $^g \wp$, where $^g \wp$ denotes the image of $\wp$ under the action of $g \in G$.

\begin{prop} \label{prop:properties of xi}
$\Xi \col \mathcal{P} \to \Acl$ has the following properties:
\begin{itemize}
\item[(i)] For any $\wp \in \mathcal{P}$ and $g \in G$ we have $\Xi(^g \wp) = \Xi(\wp)$.
\item[(ii)] Let $\wp_1, \wp_2$ be two paths of edges such that the last edge in $\wp_1$ is  $e_1$, the first edge of $\wp_2$ is $e_2$ such that $\tau(e_1)=\iota(e_2)$ and such that $e_1 \not = e_2^{-1}$.
Then $\Xi(\wp_1 \cdot \wp_2) = \Xi(\wp_1) \Xi(e_1, e_2) \Xi(\wp_2)$ as reduced words, where $\wp_1 \cdot \wp_2$ denotes the concatenation of paths.
\item[(iii)] Let $g \in G$ and let $\wp(g)$ be the unique path of edges from one of edges $\{ (A,B), (B, A) \}$ to one of the edges $\{ (g A, g B), (g B, g A) \}$. Then $\Xi(\wp(g)) = \Phi(g)$, for $\Phi$ as above. 
\end{itemize}
\end{prop}

\begin{proof}
To see $(i)$ note that for any path $(e_1, e_2)$ with $e_1 = (g_1 A, g_1 B)$ and $e_2 = (g_2 B, g_2 A)$ we have
\[
\Xi(e_1, e_2) = \bt^{\sign(g_1^{-1} g_2)} = \bt^{\sign(g_1^{-1} g^{-1} g g_2)} = \Xi(^g(e_1,e_2))
\]
and the same argument holds for paths with $e_1=(g_1 B, g_1 A)$ and $e_2 = (g_2 A, g_2 B)$.
Point $(ii)$ is immediate from the definition.

To see $(iii)$, without loss of generality assume that the normal form of $g$ is $g = a_1 b_1 \cdots a_k b_k$.
Then 
\[
\wp(g) = (B, A),(a_1 A, a_1 B),(a_1 b_1 B, a_1 b_1 A), \ldots, (g B, g A)
\]
and comparing $\Xi(\wp(g))$ with $\Phi(g)$ yields $(iii)$.
\end{proof}

We can now prove Lemma \ref{lemma:amalgamated yields letter-quasimorphism}:
\begin{proof}
Let $g,h \in G$.
First, suppose that the midpoints of 
\begin{align} \label{equ:midpoint}
\{ (A,B), (B,A) \} \text{, } \{ (gA,gB), (gB,gA) \} \text{ and } \{ (ghA,ghB), (ghB,ghA) \}
\end{align}
lie on a common geodesic segment in $\Trm$.
If the midpoint of $\{ (gA,gB), (gB,gA) \}$ lies in the middle of this segment then there are paths $\wp_1$ and $\wp_2$ such that
$\wp(g) = \wp_1 \cdot e$, $^g \wp(h) = e \cdot \wp_2$ and $\wp(gh) = \wp_1 \cdot e \cdot \wp_2$ for $e$ either $(gA, gB)$ or $(gB, gA)$.
We see that in this case
$\Xi(\wp_1 \cdot e) \cdot \Xi(e \cdot \wp_2) = \Xi(\wp_1 \cdot e \cdot \wp_2)$ as reduced words in $\Acl$ and hence
$\Phi(g) \Phi(h) = \Phi(gh)$.
Analogously we see that $\Phi(g) \Phi(h) = \Phi(gh)$ when the midpoint of $\{ (A,B), (B,A) \}$ or
$\{ (ghA,ghB), (ghB,ghA) \}$ lies in the middle of this segment.
Hence in this case $\Phi$, $g,h \in G$ are as in $(1)$ of Definition \ref{defn:letter quasihomomorphism}.

Now suppose that the midpoints in  (\ref{equ:midpoint}) do not lie on a common geodesic segment.
Then there are non-trivial paths $\wp_1, \wp_2, \wp_3 \in \mathcal{P}$ with initial edges $e_1, e_2, e_3$ satisfying
$\iota(e_1)=\iota(e_2)=\iota(e_3)$ and $e_i \not = e_j$ for $i \not = j$ such that
\[
\wp(g) = \wp_1^{-1} \cdot \wp_2 \text{ ,  }  ^g \wp(h) = \wp_2^{-1} \cdot \wp_3 \text{ , and } ^{gh} \wp(((gh)^{-1}) = \wp_3^{-1} \cdot \wp_1.
\]
By Proposition \ref{prop:properties of xi} we infer that
\begin{align*}
\Phi(g) &= c_1^{-1} \Xi(e_1^{-1}, e_2) c_2 \\
\Phi(h) &= c_2^{-1} \Xi(e_2^{-1}, e_3) c_3 \\
\Phi(gh)^{-1} &= c_3^{-1} \Xi(e_3^{-1}, e_1) c_1
\end{align*}
for $c_i = \Xi(p_i)$, $i \in \{1,2,3 \}$. Without loss of generality assume that $e_i = (g_i A, g_i B)$, the case $e_i = (g_i B, g_i A)$ is analogous. Then
\begin{align*}
\Phi(g) &= c_1^{-1} \xt_1 c_2 \\
\Phi(h) &= c_2^{-1} \xt_2 c_3 \\
\Phi(gh)^{-1} &= c_3^{-1} \xt_3 c_1
\end{align*}
\[
\xt_1= \bt^{\sign_B(g_1^{-1} g_2)} \text{, } \xt_2 = \bt^{\sign_B(g_2^{-1} g_3)} \text{, and } \xt_3 = \bt^{\sign_B(g_3^{-1} g_1)}
\]
We claim that $\sign_B(g_1^{-1} g_2) + \sign_B(g_2^{-1} g_3) + \sign_B(g_3^{-1} g_1) \in \{ -1, +1 \}$.
To see this, note that all of the signs are either $\{ +1, -1 \}$ as the edges $e_i$ were assumed to be distinct.
Suppose that 
$\sign_B(g_1^{-1} g_2)= \sign_B(g_2^{-1} g_3)= \sign_B(g_3^{-1} g_1)=1$.

Then $g_1^{-1} g_2 C \succ C$, hence $g_3^{-1} g_2 C = (g_3^{-1} g_1) g_1^{-1} g_2 C \succ g_3^{-1} g_1 C \succ C$, so $\sign_B(g_3^{-1} g_2) = 1$ and hence $\sign_B(g_2^{-1} g_3)=-1$, contradiction. 
Similarly, not all signs can be negative. Hence indeed $\sign_B(g_1^{-1} g_2) + \sign_B(g_2^{-1} g_3) + \sign_B(g_3^{-1} g_1) \in \{ -1, +1 \}$ and so $\xt_1 \xt_2 \xt_3 \in \{ \bt, \bt^{-1} \}$. This shows that $\Phi$ is as in $(2)$ of Definition \ref{defn:letter quasihomomorphism}, hence $\Phi$ is a letter-quasimorphism.
\end{proof}

\begin{thm} \label{thm:amalgamation}
Let $A, B, C$ be groups and $\kappa_A \col C \inj A$, $\kappa_B \col C \inj B$ be injections such that both $\kappa_A(C)$ and $\kappa_B(C)$ are  left relatively convex subgroup of $A$ resp. $B$.
Let $G = A \star_C B$ be the amalgamated free product for this data.
Then for every element $g_0 \in G$ which does not conjugate into $A$ or $B$, there is a homogeneous quasimorphism $\bar{\phi} \col G \to \R$ such that $\bar{\phi}(g_0) \geq 1$, $D(\bar{\phi}) \leq 1$ and $\bar{\phi}$ vanishes on $A$ and $B$.
If $g_0 \in G'$, then $scl(g_0) \geq 1/2$.

If $G$ is countable then there is an action $\rho \col G \to \mathrm{Homeo}^+(S^1)$ such that $[\delta^1 \bar{\phi}]=\rho^*\eurm^\R_b \in \Hrm^2_b(G,\R)$, for $\eurm^\R_b$ the real bounded Euler class.
\end{thm}

\begin{rmk} \label{rmk:chen-heuer}
The methods developed in this paper may be modified to obtain similar gap results for HNN-extensions and graphs of groups, as well gap results for certain one-relator groups. A generalisation of this and direct proofs of these results using both quasimorphisms and surface mappings will appear in the forthcoming preprint \cite{chen_heuer}.
\end{rmk}

The existence of a uniform gap was known before; see \cite{calegari_fujiwara} and Subsection \ref{subsec:spectral gaps of scl}.
\begin{proof} 
Let $g_0 \in G$ be as in the Theorem. Then, if $g_0$ does not conjugate into $A$ or $B$ we may conjugate $g_0$ by an element $g_1 \in G$ such that
\[
g' = g_1 g_0 g_1^{-1} = a_1 b_1 \cdots a_k b_k
\]
for \emph{all of} $a_i \in A \smallsetminus \kappa_A(C)$ and $b_i \in B \smallsetminus \kappa_B(C)$.
It follows that $\Phi(g')=w$ is a non-empty alternating word of even length and that 
$\Phi({g'}^n) = w^n$ for $n \in \N$.
By Theorem \ref{thm:main} there is a homogeneous quasimorphism $\bar{\phi} \col G \to \R$ with
$D(\bar{\phi}) \leq 1$ and $1 \leq \bar{\phi}(g_0) = \bar{\phi(g')}$ using that homogeneous quasimorphisms are invariant under conjugation.
If $G$ is countable then this quasimorphism $\bar{\phi}$ is moreover induced by a circle action
$\rho \col G \to \Homeorm^+(S^1)$.
\end{proof}

\section{Right-Angled Artin Groups} \label{sec:RAAGs and scl}
In this section all graphs will be simplicial, i.e. do not contain multiple edges between two vertices or loops.
Let $\Gamma$ be a finite simplicial graph with vertices  $\Vrm(\Gamma)$ and edges $\Erm(\Gamma)$.
Given a subset $\Lambda \subset \Vrm(\Gamma)$ the \emph{full subgraph on $\Lambda$ in $\Gamma$} is the graph with vertices $\Lambda$ where two elements $v,w \in \Lambda$ are connected by an edge if and only if they are connected in $\Gamma$.

For a vertex $v \in \Gamma$, the \emph{link of $v$} is the full subgraph of the set $\{w \mid (v,w) \in \Erm(\Gamma) \}$ in $\Gamma$ and denoted by 
$\textrm{Lk}(v)$. The \emph{closed star} is the full subgraph of $\textrm{Lk}(v) \cup \{ v \}$ in $\Gamma$ and denoted by $\textrm{St}(v)$. 
The \emph{right-angled Artin group} or \emph{RAAG} on $\Gamma$ is the group $\Arm(\Gamma)$ with group presentation
\[
\Arm(\Gamma) = \langle \Vrm(\Gamma) \mid [v, w]; (v,w) \in \Erm(\Gamma) \rangle
\]
A word $w$ in the generators $\Vrm(\Gamma)$ representing an element $[w] \in \Arm(\Gamma)$ is called \emph{reduced} if it has minimal word length among all words representing $[w]$. A word $w$ is said to be cyclically reduced if it has minimal word length among all of its conjugates. 
The \emph{support} of an element $g \in \Arm(\Gamma)$ is the set of vertices that appear in a reduced word representing $g$. It is well-known that the support is well-defined.

Let $\Gamma$ be a finite simplicial graph, let $\Arm(\Gamma)$ be the right-angled Artin group of $\Gamma$ and let $v  \in \Gamma$. Then $\Arm(\Gamma)$ can be thought of as an amalgamated free product of $\Arm(\textrm{St}(v))$ and $A(\Gamma \backslash \{ v \} )$ where the common subgroup is $\Arm(\textrm{Lk}(v))$. i.e.
\[
\Arm(\Gamma) = \Arm(\textrm{St}(v)) \star_{\Arm(\textrm{Lk}(v))} \Arm(\Gamma \backslash \{v \}).
\]
This will be used both in the proof of Theorem \ref{thm:raags and scl} and for induction arguments. 
\begin{prop} \label{prop: convex subgroups of raartin groups}
(Section 4 of \cite{convexsub}) Let $\Lambda \subset \Gamma$ be a full subgraph of $\Gamma$. Then $\Arm(\Lambda) < \Arm(\Gamma)$ induced by the embedding, is a left relatively convex subgroup. 
\end{prop}

\begin{proof}
We follow the proof of \cite{convexsub}.
We may induct on the following statement: For any $\Gamma$ of size at most $k$ and every full subgraph $\Lambda \subset \Gamma$, $\Arm(\Lambda)$ is left relatively convex in $\Arm(\Gamma)$. For $k=2$ this is just the case of free-abelian and non-abelian free groups mentioned before. 
Assume the statement is true for all $n \leq k$. Let $\Gamma$ be a graph with $k+1$ vertices and let $\Lambda \subset \Gamma$ be a full subgraph. If $\Lambda = \Gamma$ there is nothing to show. Else pick $v \in \Vrm(\Gamma) \backslash \Vrm(\Lambda)$ and set $\Gamma'$ to be the full subgraph in $\Gamma$ on the vertices $\Vrm(\Gamma) \backslash \{ v \}$.
Hence $\Lambda \subset \Gamma' \subset \Gamma$ with $\Gamma'$ of size $k$.
We wish to show that $\Arm(\Gamma') < \Arm(\Gamma)$ is a left-relatively convex subgroup.
Consider the amalgamation
\[
\Arm(\Gamma) = \Arm(\textrm{St}(v)) \star_{\Arm(\textrm{Lk}(v))} \Arm(\Gamma')
\]
By induction, $\Arm(\textrm{Lk}(v)) < \Arm(\Gamma')$ is a left relatively convex subgroup.
Also $\Arm(\textrm{Lk}(v)) < \Arm(\textrm{St}(v))$ is a left relatively convex subgroup
as $\Arm(\textrm{St}(v)) = \langle v \rangle \times \Arm(\textrm{Lk}(v))$. 

We may use Corollary \ref{corr:orders of amalgamations} to see that $\Arm(\Gamma') < \Arm(\Gamma)$ is a left relatively convex subgroup. By induction hypothesis, $\Arm(\Lambda) < \Arm(\Gamma')$ is a left-relatively convex subgroup and by transitivity $\Arm(\Lambda) < \Arm(\Gamma')$ is a left relatively convex subgroup. 
\end{proof}

We deduce:
\begin{thm} \label{thm:quasimorphisms on raags}
Let $g \in \Arm(\Gamma)$ be an element in an right-angled Artin group $\Arm(\Gamma)$ such that $g_0$ does not conjugate into a subgroup of a clique of $\Gamma$. Then there
is a homogeneous quasimorphism $\bar{\phi}$ which vanishes on the generators $\Vrm(\Gamma)$ such that $\bar{\phi}(g_0) \geq 1$ and $D(\bar{\phi}) \leq 1$. 

Moreover, there is an action $\rho \col \Arm(\Gamma) \to \mathrm{Homeo}^+(S^1)$ such that $[\delta^1 \bar{\phi}]=\rho^*\eurm^\R_b \in \Hrm^2_b(G,\R)$, for $\eurm^\R_b$ the real bounded Euler class.
\end{thm}
Observe that no non-trivial element in the commutator subgroup of a right-angled Artin group conjugates into a clique. 
An application of Bavard's Duality Theorem \ref{thm:Bavards duality} yields:
\begin{thm} \label{thm:raags and scl}
Let $g_0$ be a non-trivial element in the commutator subgroup of a right-angled Artin group. Then $\scl(g_0) \geq 1/2$. This bound is sharp. 
\end{thm}

\begin{proof}
(of Theorem \ref{thm:quasimorphisms on raags})
Let $g \in \Arm(\Gamma)$ be such an element. We may suppose that $g$ is cyclically reduced, as homogeneous quasimorphisms are invariant under conjugation. 
Choose a vertex $v$ in the support of $g$ such that there is another vertex $w$ in the support of $g$ which is non-adjacent to $v$. Such a vertex exists as $g$ does not conjugate into a clique. Write $\Arm(\Gamma)$ as
\[
\Arm(\Gamma) = \Arm(\textrm{St}(v)) \star_{\Arm(\textrm{Lk}(v))} \Arm(\Gamma \backslash \{v \})
\]
and observe that $g$ does not conjugate into any factor of this amalgamation as both $v$ and $w$ are in the support of $g$.
By Proposition \ref{prop: convex subgroups of raartin groups}, both $\Arm(\textrm{Lk}(v)) < \Arm(\textrm{St}(v))$ and 
$\Arm(\textrm{Lk}(v)) < \Arm(\Gamma \backslash \{v \})$ are left relatively convex subgroups.
We conclude using Theorem \ref{thm:amalgamation}. Commutators in $\Arm(\Gamma)$ have $\scl$ at most $1/2$. Hence this bound is sharp.
\end{proof}

\bibliographystyle{alpha}
\bibliography{bib_raagsandscl}

\end{document}